\documentclass[12pt,a4paper]{article}
\usepackage{amsmath, enumerate, bm, geometry, xspace, color}
\usepackage{amsthm}
\usepackage{amssymb,tikz}
\numberwithin{equation}{section}
\theoremstyle{plain}
\newtheorem{thm}{Theorem}[section]
\newtheorem{prop}[thm]{Proposition}
\newtheorem{cor}[thm]{Corollary}
\newtheorem{lem}[thm]{Lemma}
\theoremstyle{definition}
\newtheorem{exa}[thm]{Example}

\newtheorem{rem}[thm]{Remark}
\newtheorem{defi}[thm]{Definition}
\newtheorem{prob}[thm]{Problem}

\newcommand{\mc}[1]{\mathcal{#1}}

\newcommand{\R}{\mathbb{R}}
\newcommand{\C}{\mathbb{C}}
\newcommand{\N}{\mathbb{N}}
\newcommand{\Z}{\mathbb{Z}}
\newcommand{\T}{\mathbb{T}}
\newcommand{\disc}{\mathbb{D}}

\newcommand{\An}{A}                    
\newcommand{\Ad}{\mathbf{A}}     
\newcommand{\bE}{\mathbb{E}}    
\newcommand{\cP}{\mathcal{P}}    
\newcommand\cL{\mathcal{L}}   
\newcommand{\Tr}{{\rm Tr}}        
\newcommand{\id}{{\rm 1}}            
\newcommand{\Id}{{\rm id}}            
\newcommand{\D}{\mathbf{D}}       
\newcommand{\Rot}{\mathbf{R}}       
\newcommand{\Law}{\mathcal{L}}  
\newcommand{\ep}{\varepsilon}     
\newcommand{\ii}{{\rm i}}               
\renewcommand{\Re}{\hspace{0.3mm}{\rm Re}}
\renewcommand{\Im}{\hspace{0.3mm}\text{\normalfont Im}}
\newcommand{\supp}{{\rm supp}} 
\newcommand{\Ind}{\mathbf{1}}   

\newcommand{\fD}{\mathfrak{D}}   
\newcommand{\M}{\mathbb{M}}    
\newcommand{\cE}{\mathcal{E}}    
\newcommand{\vv}{v}                    
\newcommand{\AS}{{\rm(AS)}\xspace}  
\newcommand{\ff}{\frac{\alpha t}{\alpha + (1-\alpha)t}}
\newcommand{\ID}{\mc{ID}}    

\newcommand{\bfC}{{\bf C}}    
\newcommand{\bfF}{{\bf F}}     

\newcommand{\bnu}{{\bm\nu}}       
\newcommand{\bfs}{\mathbf{s}}     
\newcommand{\bff}{\mathbf{f}}     
\newcommand{\bfc}{\mathbf{c}}    
\newcommand{\bfb}{\mathbf{b}}     
\renewcommand{\DH}{{\rm D \hspace{-0.5mm}H}}    
\newcommand{\Haar}{{\bf h}}  

\makeatletter
  \def\mathcomposite{%
     \@ifstar
        {\def\@mathcomposite@option{%
            \baselineskip\z@skip\lineskiplimit-\maxdimen}%
         \@mathcomposite}%
        {\let\@mathcomposite@option\offinterlineskip
         \@mathcomposite}}
  \def\@mathcomposite{%
     \@ifnextchar[\@@mathcomposite{\@@mathcomposite[0]}}
  \def\@@mathcomposite[#1]#2#3#4{%
     #2{\mathchoice
        {\@mathcomposite@{#1}{#3}{#4}\displaystyle{1}}%
        {\@mathcomposite@{#1}{#3}{#4}\textstyle{1}}%
        {\@mathcomposite@{#1}{#3}{#4}%
         \scriptstyle\defaultscriptratio}%
        {\@mathcomposite@{#1}{#3}{#4}%
         \scriptscriptstyle\defaultscriptscriptratio}}}
  \def\@mathcomposite@#1#2#3#4#5{%
     \vcenter{\m@th\@mathcomposite@option
        \dimen@\f@size\p@\dimen@#1\dimen@\dimen@#5\dimen@
        \divide\dimen@ 18
        \edef\@mathcomposite@skipamount{\the\dimen@}%
        \ialign{\hfil$#4##$\hfil\cr
           #2\crcr
           \noalign{\vskip\@mathcomposite@skipamount}%
           #3\crcr}}}
  \makeatother

\newcommand{\utimes}{\kern-0.1ex\mathcomposite[-12]{\mathrel}{\cup}{\times}}                         
\newcommand{\sutimes}{\kern-0.1ex\mathcomposite[-13.2]{\mathrel}{\cup}{\times}\kern-0.1ex}    

\def\tolaw{\overset{\rm law}{\longrightarrow}}    
\def\eqlaw{\overset{\rm  law}{=}}                       
\def\wto{\overset {\rm w}{\rightarrow}}               

\begin{document}
\title{Limit theorems for free L\'evy processes}

\author{Octavio Arizmendi and Takahiro Hasebe}   
\date{}
\maketitle
\abstract{We consider different limit theorems for additive and multiplicative free L\'evy processes.  The main results are concerned with positive and unitary multiplicative free L\'evy processes at small time, showing convergence to log free stable laws for many examples. The additive case is much easier, and we establish the convergence at small or large time to free stable laws. During the investigation we found out that a log free stable law with index $1$ coincides with the Dykema-Haagerup distribution. We also consider limit theorems for positive multiplicative Boolean L\'evy processes at small time, obtaining log Boolean stable laws in the limit. }
\tableofcontents

\section{Introduction}
\subsection{Background}

This article investigates the asymptotic behavior of additive and multiplicative free L\'evy processes (AFLP and MFLP, resp.) at small time and large time. 
These are the free analogs of L\'evy processes and were introduced by Biane \cite{Bia98} as particular cases of processes with free increments. There are two possibilities for doing these, depending on weather one considers stationary free increments or stationary Markov transition functions. We will only consider the former ones, since they are related directly to convolution semigroups for the free convolutions.
 
 In this setting there are various interesting questions which naturally appear as analogs of classical results. However, in the free world the answer to this questions sometimes are similar and sometimes can be quite different to the classical. 
  
The first question that we investigate is the following.  Given an AFLP $\{X_t\}_{t\geq0}$ such that $X_0=0$, when does the convergence in law of the process 
\begin{equation}\label{AFLP1}
a(t) X_t +b(t), \quad \text{as} \quad t\downarrow0  \quad \text{or} \quad t\to\infty, 
\end{equation} 
holds for some functions $a\colon (0,\infty) \to (0,\infty)$ and $b\colon (0,\infty) \to \R$?  
This problem can be settled by the Bercovici-Pata bijection, and the result has one-to-one correspondence with the classical case (see Section \ref{sec ALP}). In both cases of small time and large time, 
the set of limiting distributions is exactly the set of free stable distributions (Proposition \ref{prop stable}). It is notable that, in classical probability, while limit theorems for sums of independent random variables (discrete time case) have been well studied around 1930's and 1940's \cite{GK54}, limit theorems for L\'evy processes (continuous time case) were only settled rather recently by Maller and Mason \cite{MM08,MM09}. 

The second question that can be considered concerns, given a positive MFLP such that $X_0=\id$, the convergence in law of 
\begin{equation}\label{MFLP1}
b(t) (X_t)^{a(t)}, \quad \text{as} \quad t\to\infty, 
\end{equation}
where $a,b\colon (0,\infty) \to (0,\infty)$ are some functions. This problem was solved by Haagerup-M\"{o}ller \cite{HM}, following previous results of Tucci \cite{T10}. 
The set of possible limit distributions is completely known (see Section \ref{sec:Haagerup-Moeller}). In fact, for every positive MFLP $\{X_t\}_{t\geq0}$, the law of the process $(X_t)^{1/t}$ converges weakly to a probability measure $\nu$, and this map $\{X_t\}_{t\geq0} \mapsto \nu$  (more precisely, the map $\Law(X_1) \mapsto \nu$, where $\Law(X_1)$ is the law of $X_1$) is injective. This result is quite different from classical probability (see Proposition \ref{multiplicativeclassical}) where the limit distributions must be {\it log stable distributions},  which are push-forward of stable distributions by the map $x\mapsto e^x$. This terminology is adopted for other distributions as well, e.g.\ log Cauchy distributions. 
Note that in classical probability, additive and multiplicative classical L\'evy processes (ACLP and MCLP, resp.)\ can be identified by the exponential map, so one need not study MCLPs. However, due to the non-commutativity of processes, MFLPs cannot be identified with AFLPs by the exponential map. 

The third question to be considered is the limit in law of \eqref{MFLP1} at small time, namely 
\begin{equation}\label{MFLP2}
b(t) (X_t)^{a(t)}, \quad \text{as} \quad t\downarrow0, 
\end{equation}
where $\{X_t\}_{t\geq0}$ is a positive MFLP starting at $\id$, and $a,b\colon (0,\infty) \to (0,\infty)$ are some functions as before. However, as we will see, the situation is very different than for the large $t$ limit. The mains results in this direction are summarized in Section \ref{sec:main results}. 

A similar question one can consider is the limit distribution of 
\begin{equation}\label{MFLP3}
b(t) (U_t)^{a(t)},  \quad \text{as} \quad t\downarrow0, 
\end{equation}
where $\{U_t\}_{t\geq0}$ is a unitary MFLP such that $U_0=\id$ and $a\colon (0,\infty) \to \Z$ and $b\colon (0,\infty) \to \T$ are some functions. The function $a$ should take only integral values, since a power function $z^p$ is continuously defined on the unit circle only when $p$ is an integer. Notice that in this case we should talk about small time limits, since for large time, the distribution of $U_t$ spreads and hence we have to require $a(t) \to 0$ to get a non-Haar measure in the limit, but then $a(t) \equiv 0$ eventually. 

In other directions, we also consider the Boolean analogues of the processes \eqref{AFLP1} and \eqref{MFLP2}.  In the Boolean case we cannot talk about large time limits \eqref{MFLP1} since in generic cases positive multiplicative Boolean L\'evy processes (MBLP) do not exist at large time \cite{Ber06}. We do not analyze the unitary case in this paper.

We shall mention that another direction of study, not discussed in this paper, is the limit theorem for positive multiplicative monotone LPs, both when $t\to\infty$ and $t\downarrow0$ and also unitary ones when $t\downarrow0$. Also, even additive monotone LPs have some open problems (see Remark \ref{rem:monotone}). These problems are left to future research.

\subsection{Main results}\label{sec:main results}
Our main results are summarized in the following list. 

\begin{enumerate}[\rm(1)]
\item\label{Main1} The set of possible limit distributions of processes of the form \eqref{MFLP2} contains the following distributions:   
\begin{itemize} 
\item the log free stable distributions with index $1$, which contain log Cauchy distributions and the Dykema-Haagerup distribution (Theorems \ref{free Bessel} and \ref{limit} and Corollary \ref{cor DH}); 
\item some log free $\alpha$-stable distributions with $1 < \alpha \leq2$ (Theorem \ref{multi log FS} and Corollary \ref{cor LFS}). 
\end{itemize}
Moreover, we provide a general sufficient condition on $\{X_t\}_{t\geq0}$ and on functions $a$ and $b$ such that the law of \eqref{MFLP2} converges to the log Cauchy distribution (Theorem \ref{LCF}). 

\item\label{Main3} The set of possible limit distributions of \eqref{MFLP2}, now with $\{X_t\}_{0 \leq t \leq1}$ a positive MBLP, contains the log Boolean stable distributions with index $\leq1$. We provide a general condition on $\{X_t\}_{t\geq0}$, functions $a$ and $b$ such that the process converges in law (Theorems \ref{LCB} and \ref{LBS}). 

\item\label{Main4} The set of possible limit distributions of processes of the form \eqref{MFLP3} contains all ``wrapped free stable distributions'', which are the distributions of random variables $e^{\ii X}$ where $X$ follows a free stable law (Corollary \ref{cor UMFLP}). We also give a fairly large domain of attraction of a wrapped free stable distribution (Theorem \ref{thm UMFLP0}). A similar result is obtained for unitary MCLPs, which seems unknown in the literature. 
\end{enumerate}

Before going into the proofs, we would like to make some comments regarding the above results.

The Dykema-Haagerup distribution mentioned in \eqref{Main1} was introduced in \cite{DH04a} and it appeared as the limiting eigenvalue distribution of $T_N^* T_N$ where $T_N$ is an $N\times N$ upper-triangular random matrices with independent complex Gaussian entries. During our investigation of \eqref{Main1}, we discovered a mysterious fact that the Dykema-Haagerup distribution coincides with a log free $1$-stable law (Proposition \ref{prop DH}). 

One observation on the result \eqref{Main1} is that the limit distributions for positive MFLPs at small time seem to be universal, in contrast to the non-universal limit distributions of MFLPs at large time.

The proof of \eqref{Main1} is mostly based on the moment method. We find explicit MFLPs $\{X_t\}_{t\geq0}$ and explicit functions $a$ and $b$ such that the moments of \eqref{MFLP2} converge. A particularly strong result can be obtained for the convergence to log Cauchy distributions (Theorem \ref{LCF}). In this case we can reduce the problem to the Boolean case \eqref{Main3}, which is rather easy to analyze. This reduction procedure, however, needs a considerable generalization of free and boolean convolutions beyond probability measures, which we prepare in Section \ref{sec6}. 
After all our investigation, it remains open whether the set of possible limit distributions of \eqref{MFLP2} is exactly the set of all log free stable distributions.  

The term MBLP in \eqref{Main3} is not very rigorous since it is only defined in the sense of a convolution semigroup of distributions, and no operator model is known. Also, the convolution semigroup is only defined for time $t \in [0,1]$ in general. The proof of \eqref{Main3} is easier than the free case \eqref{Main1} and a more solid result can be proved. Thanks to a simple formula for multiplicative Boolean convolution, we can directly compute the density of the process, and show that it converges to the density of log Boolean stable distributions.

For the unitary case \eqref{Main4}, it again remains open whether the set of limit distributions of \eqref{MFLP3} is exactly the set of push-forwards of free stable distributions by the exponential map $x \mapsto e^{\ii x}$. 
The proof of \eqref{Main4} uses the (clockwise) exponential map $x \mapsto e^{-\ii x}$, in order to reduce unitary MFLPs to AFLPs. In spite of the non-commutativity of the process, such a reduction is possible, thanks to the work of Anshelevich and Arizmendi \cite{AA}. This method of using the exponential map has been limited to multiplicative convolutions on $\T$ so far, and not available to positive multiplicative convolutions, and hence not available to \eqref{Main1}.

\subsection{Organization of the paper}

Apart from this introduction there are six sections.  We introduce notations and preliminaries needed for the subsequents section in Section \ref{S2}. This includes standard background in free probability but also some useful lemmas on convergence of measures and the exponential map. In Section \ref{S3} we present, for completeness,  results which are known or which follow directly from other known results, including limit theorems for additive free L\'evy processes.  The main results are in the rest of the sections. More specifically, in Section \ref{sec PMFLP} we consider positive MFLPs at small time. This section is mostly devoted to give many examples of families for which we can prove convergence to log free stable distributions. The general result for log Cauchy is separated as Section \ref{sec6} since, on one hand, the proof is rather technical, and on the other hand, we introduce a class of generalized $\eta$-transforms which may be helpful in other problems in future. Section \ref{S5} is devoted to the positive MBLPs. Finally, in Section \ref{S7} we use the exponential map to study unitary MCLPs and MFLPs.

\section{Preliminaries} \label{S2}

\subsection{Notation}
\begin{enumerate}


\item $\C^+, \C^-$: the upper and lower half-planes of the complex plane $\C$, respectively.   

\item $\T$: the unit circle $\{z \in \C: |z| =1 \}$. 

\item $\disc$: the open unit disc $\{z \in \C: |z| <1 \}$. 

\item $\cP(T)$: the set of Borel probability measures on a topological space $T$. 

\item $\Law(X)$: the law of a random variable $X$ taking values in $\R$ or $\T$. 

\item $\mu^p, p \in\R$: the push-forward of a probability measure $\mu$ on $(0,\infty)$ by the map $x\mapsto x^p$. If $\mu$ is a probability measure on $[0,\infty)$ then we can define $\mu^p$ for $p \geq0$, and if $\mu$ is a probability measure on $\T$ then we define $\mu^n$ for $n\in\Z$.  

\item $\D_s(\mu), s \in \R$: the dilation of a probability measure $\mu$, that is, the push-forward of $\mu$ induced by the map $x\mapsto s x$. 

\item $\Rot_w(\mu), w \in \T$: the rotation of a probability measure $\mu$ on $\T$ induced by the map $z \mapsto w z$.

\item $z^\alpha, \log z$: the principal value unless specified otherwise. 

\end{enumerate}

\subsection{Classical convolution}
Recall that the classical convolution $\mu_1\ast\mu_2$ of Borel probability measures $\mu_1$ and $\mu_2$ on $\R$ is the law of $X_1+X_2$, where $X_1$ and $X_2$ are independent, $\R$-valued random variables such that $\Law(X_i)=\mu_i, i=1,2$. Equivalently, it is characterized by  
\begin{equation}
\int_\R f(x)\,d(\mu_1 \ast \mu_2)(x) = \int_{\R^2} f(x+y) \,d\mu_1(x)d\mu_2(y)
\end{equation}
for bounded continuous functions $f$ on $\R$. A central concept of this paper is {\it infinite divisibility}, which we define in a general framework for later use. 
\begin{defi}
Suppose that $\star$ is an associative binary operation on Borel probability measures on a topological space $T$. A Borel probability measure $\mu$ on $T$ is said to be $\star$-infinitely divisible (or $\star$-ID for short) if,  for any $n \in \N$, there exists a probability measure $\mu_n$ on $T$ such that $\mu=\mu_n ^{\star n}: = \mu_n \star \cdots  \star \mu_n$. The class of such probability measures is denoted by $\ID(\star)$ or $\ID(\star, T)$. 
\end{defi}

Recall also that if a probability measure $\mu$ on $\R$ is $\ast$-ID then its characteristic function has the L\'{e}vy-Khintchine representation (see e.g.\ \cite{GK54,Sat99})
\begin{equation}
\int_\R e^{\ii x z} \,d\mu(x)=\exp\left[\ii\xi z +\int_\R\left(e^{\ii z x}-1- \frac{\ii z x}{1+x^2}\right) \frac{1+x^2}{x^2}\,\tau(dx)\right],
\text{ \ \ }z\in \R,  \label{CLK}
\end{equation}
where $\xi \in \R$ and $\tau$ is a nonnegative finite Borel measure on $\R$. Conversely, given such a pair $(\xi,\tau)$, the RHS of \eqref{CLK} is the characteristic function of a $\ast$-ID distribution. The pair $(\xi,\tau)$ is unique and is called the {\it (additive) classical generating pair} of $\mu$. 
We denote by $\mu^{\xi,\tau}_\ast$ the $\ast$-ID distribution which has the classical generating pair $(\xi,\tau)$.
For each $\ast$-ID distribution $\mu$, there exists an ACLP $\{X_t\}_{t\geq0}$ such that $X_0=0$ and $\Law(X_1)=\mu$ (see \cite{Sat99}). Then the law of $X_t$ is denoted by $\mu^{\ast t}$, which is characterized by the generating pair $(t\xi, t \tau)$ and which forms a convolution semigroup, $\mu^{\ast s} \ast \mu^{\ast t}= \mu^{\ast (s+t)}$ for $s,t\geq0$.

\subsection{Free convolution} 
The {\it (additive) free convolution} $\mu_1 \boxplus \mu_2$ of (Borel) probability measures $\mu_1$ and $\mu_2$ on $\R$ is the distribution of noncommutative random variable $X_1+X_2$, where  $X_1, X_2$ are free selfadjoint random variables in some non-commutative probability space such that $\Law(X_1)=\mu_1$ and $\Law(X_2)=\mu_2$.  
Free convolution was first defined by Voiculescu \cite{Voi85,Voi86} for compactly supported distributions and then generalized by Maassen \cite{Maa92} for probability measures with finite variances, and finally by Bercovici-Voiculescu \cite{BV93} for arbitrary ones. 

Given a probability measure $\mu$ on $\R$, let 
\begin{equation}
G_\mu(z) = \int_{\R}\frac{\mu(dx)}{z-x},\qquad F_\mu(z)=\frac{1}{G_\mu(z)},\qquad z\in \C\setminus \R,
\end{equation} 
be the {\it Cauchy transform} and the \textit{reciprocal Cauchy transform (or $F$-transform)} of $\mu$, respectively. For $\alpha,\beta>0$, let 
$\Gamma_{\alpha,\beta}$ be the truncated cone 
\begin{equation}
\{z\in\C^+: |\Re(z)| < \alpha\Im(z), |z|>\beta\}.
\end{equation}
Bercovici and Voiculescu \cite{BV93} showed that for any $\alpha>0$, there exist $\beta,\alpha',\beta'>0$ such that $F_\mu$ is univalent in $\Gamma_{\alpha',\beta'}$ and $F_\mu(\Gamma_{\alpha',\beta'}) \supset \Gamma_{\alpha,\beta}$. Hence the right compositional inverse $F_\mu^{-1}$ of $F_\mu$ may be defined in $\Gamma_{\alpha,\beta}$. 
The \textit{Voiculescu transform} of $\mu$ is then defined by 
\begin{equation}
\varphi _{\mu }\left(z\right) =F_{\mu }^{-1}(z)-z,\qquad z \in \Gamma_{\alpha,\beta}. 
\end{equation}
The Voiculescu transform characterizes free convolution:  given two probability measures $\mu_1,\mu_2$ on $\R$,  the identity
\begin{equation}
\varphi_{\mu_1\boxplus\mu_2}(z) = \varphi_{\mu_1}(z) + \varphi_{\mu_2}(z)
\end{equation}
holds in the intersection of the domains of $\varphi_{\mu_1}, \varphi_{\mu_2}$ and $\varphi_{\mu_1\boxplus\mu_2}$.

A $\boxplus$-ID measure has a free analogue of the L\'{e}vy-Khintchine representation. 
\begin{thm}[Bercovici-Voiculescu \cite{BV93}] \label{thmBV93}
Let $\mu$ be a probability measure on $\R$. The following are equivalent. 
\begin{enumerate}[\quad\rm(1)] 
\item $\mu$ is $\boxplus$-ID. 
\item For any $t>0$, there exists a probability measure $\mu^{\boxplus t}$ satisfying $\varphi_{\mu^{\boxplus t}}(z) = t\varphi_\mu(z).$
\item There exist $\xi \in \R$ and a nonnegative finite Borel measure $\tau$ on $\R$ such that 
\begin{equation}
\varphi_{\mu}(z)=\xi+\int_\R \frac{1+z x}{z-x}\, \tau(d x),\qquad z\in \Gamma_{\alpha,\beta}.  \label{FLK}
\end{equation}
\end{enumerate}
Conversely, given a pair $(\xi,\tau)$ of a real number and a nonnegative finite Borel measure, there exists a $\boxplus$-ID distribution $\mu$ such that \eqref{FLK} holds. The pair $(\xi,\tau)$ is unique and is called the \textit{(additive) free generating pair} of $\mu$. 
\end{thm}
We denote by $\mu_\boxplus^{(\gamma,\tau)}$ the $\boxplus$-ID distribution characterized by \eqref{FLK}. 
Similarly to classical probability, for each $\boxplus$-ID distribution $\mu$ there exists an AFLP $\{X_t\}_{t\geq0}$ (affiliated to a finite von Neumann algebra with normal faithful finite trace) such that $X_0=0$ and $\Law(X_t)=\mu^{\boxplus t}$ (see \cite{Bia98,B-NT02}). 

The bijection
\begin{equation}
\Lambda\colon  \ID(\ast) \to  \ID(\boxplus), \qquad \mu^{\xi,\tau}_\ast \mapsto \mu^{\xi,\tau}_\boxplus
\end{equation}
is called the {\it Bercovici-Pata bijection}. Moreover, this map is a homeomorphism with respect to weak convergence \cite[Corollary 3.9]{B-NT02}.

\subsection{Boolean convolution}
The \textit{Boolean convolution} $\mu_1\uplus\mu_2$ of probability measures $\mu_1$ and $\mu_2$ on $\R$ is the law of $X_1+X_2$ where $X_1$ and $X_2$ are Boolean independent selfadjoint random variables such that $\Law(X_1)=\mu_1$ and $\Law(X_2)=\mu_2$. Boolean convolution  was introduced by  Speicher and Woroudi \cite{SW97} for compactly supported probability measures and then by Franz \cite{Fra09a} for arbitrary ones.  

Boolean convolution is characterized by 
\begin{equation}
\eta_{\mu_1\uplus\mu_2}(z) = \eta_{\mu_1}(z) + \eta_{\mu_2}(z), \qquad z\in\C^-, 
\end{equation} 
where   
\begin{equation}\label{eta-f}
\eta_\mu(z)=1-zF_\mu\left(\frac{1}{z}\right),\qquad z\in\C^-, 
\end{equation}
which is called the {\it $\eta$-transform}. It can be proved that for any $t \geq 0$ and any probability measure $\mu$ on $\R$, there exists a probability measure $\mu^{\uplus t}$ which satisfies $\eta_{\mu^{\uplus t}}(z) = t \eta_\mu(z)$ in $\C^-$ \cite{SW97}.  This implies that every probability measure $\mu$ on $\R$ is $\uplus$-ID. 
Since $F_\mu$ is an analytic map from $\C^+$ into itself such that $F_\mu(z) = z(1+o(1))$ as $z \to \infty$ non-tangentially, it has the Pick-Nevanlinna representation 
\begin{equation}\label{PN}
F_\mu(z) = z - \xi+\int_\R \frac{1+z x}{x-z}\, \tau(d x),\qquad z\in\C^+, 
\end{equation}
where $\xi \in \R$ and $\tau$ is a nonnegative finite measure on $\R$. Conversely, if a map $F$ has the representation of the RHS of \eqref{PN}, it can be written as $F=F_\mu$ for some probability measure $\mu$. Thus we may denote by $\mu^{\xi,\tau}_\uplus$ the probability measure having the representation \eqref{PN}, and then define the bijection 
 \begin{equation}
 \Lambda_B\colon  \ID(\ast) \to  \ID(\uplus)=\cP(\R), \qquad \mu^{\xi,\tau}_\ast \mapsto \mu^{\xi,\tau}_\uplus
\end{equation}
which we call the {\it Boolean Bercovici-Pata bijection}. It can be proved that $\Lambda_B$ is a homeomorphism with respect to the weak convergence. A proof is not written in the literature but follows the free case \cite[Corollary 3.9]{B-NT02}.

\subsection{Stable distributions}\label{sec stable} 
Let $\star$ denote one of $\ast, \boxplus$ and $\uplus$. A non-degenerate probability measure $\mu$ on $\R$ is \emph{stable} (or free stable or Boolean stable, according to the choice of $\star$) if for every $a, b>0$ there exist $c>0, d \in \R$ such that 
\begin{equation}
\D_a(\mu) \star \D_b(\mu) = \D_c(\mu) \star \delta_d. 
\end{equation}
If we can always take $d=0$ then $\mu$ is said to be \emph{strictly stable}. There are several equivalent definitions of stable distributions (see \cite{Zol86}). 

Let $\Ad$ be the set of admissible parameters 
\begin{equation}\label{def_set_A}
\Ad:=\{\alpha \in (0,1], \; \rho \in [0,1]\} \cup 
\{\alpha\in(1,2], \; \rho \in [1-\alpha^{-1}, \alpha^{-1}]\}. 
\end{equation}
Up to scaling and shifts, stable distributions are classified by the admissible parameters. For $(\alpha,\rho)\in\Ad$ let $\bfs_{\alpha,\rho}$ be a classical stable distribution characterized by 
\begin{equation}\label{def_g_alpha_rho}
\int_{\R} e^{x z}  \,d\bfs_{\alpha,\rho}(x)=
\begin{cases}
\exp\left(- \frac{1}{\Gamma(1+\alpha)}e^{\ii \alpha \rho \pi}z^{\alpha}  \right), &\alpha\neq1, \\
\exp\left(-\ii \rho \pi z + (1-2\rho) z\log z  \right), &\alpha=1
 \end{cases}
\end{equation}
for $z\in \ii(-\infty,0)$, and let $\bff_{\alpha,\rho}$ be a free stable distribution characterized by 
\begin{equation}\label{Voiculescu}
\varphi_{\bff_{\alpha,\rho}}(z)=
\begin{cases}
-e^{\ii \alpha \rho \pi} z^{1-\alpha}, &\alpha\neq1, \\
-\ii \rho \pi - (1-2\rho) \log z, & \alpha=1,
\end{cases}
\end{equation}
for $z\in \C^+$. Note that the parametrization is changed from that of \cite{BP99}. The parameter $\rho$ expresses the mass on the positive line: $\rho = \bff_{\alpha,\rho}([0,\infty))$ if $\alpha \neq1$; see \cite{HK14}. The above free stable distributions $\bff_{\alpha,\rho}$ cover all the free stable distributions up to affine transformations; namely, the set 
$$
\{\D_a(\bff_{\alpha,\rho}) \boxplus \delta_b: a> 0, b \in \R, (\alpha,\rho) \in \Ad \}
$$ is equal to the set of free stable distributions.  Notice that the free shift $\boxplus \delta_b$ is equal to the usual shift $\ast \delta_b$. 

For notational simplicity we denote by $\bff_{\alpha}$ the free stable distribution with $\alpha \geq 1$ and $\rho =1-1/\alpha$, namely, 
$
\varphi_{\bff_{\alpha}}(z)= -(-z)^{1-\alpha} 
$
for $\alpha \in (1,2]$ and $\varphi_{\bff_{1}}(z)= - \log z$. The support of $\bff_{\alpha}$ is given by $\supp(\bff_1) =(-\infty,1]$, $\supp(\bff_\alpha) =[-\alpha(\alpha-1)^{1/\alpha-1},\infty)$ for $\alpha \in(1,2)$ and $\supp(\bff_2)=[-2,2]$. Note that $\bff_2$ is the standard semi-circular law. Further information on free stable laws is found in \cite{BP99,Dem11,HK14}.

The classical and free stable distributions are correspondent in terms of the Bercovici--Pata bijection: 
\begin{equation}
\bff_{\alpha,\rho} = \Lambda(\bfs_{\alpha,\rho}). 
\end{equation}

Boolean stable distributions are classified similarly. For later use we introduce an additional scaling parameter: 
\begin{equation}
F_{\bfb_{\alpha,\rho,r}}(z)= z + r e^{\ii \alpha \rho \pi} z^{1-\alpha}, \qquad z\in \C^+, (\alpha,\rho) \in\Ad, r>0.  
\end{equation}
The parameter $r>0$ corresponds to the convolution power and the dilation: $\bfb_{\alpha,\rho,r} =\bfb_{\alpha,\rho,1}^{\uplus r} = \D_{r^{1/\alpha}}(\bfb_{\alpha,\rho,1})$. For simplicity, we denote 
$\bfb_\alpha := \bfb_{\alpha,1,1}$ when $0<\alpha \leq1$. 
The Boolean stable laws have very explicit densities
 \begin{equation}
 \frac{d\bfb_{\alpha,\rho,r}}{dx}
 = 
 \begin{cases}
 \displaystyle \frac{ r \sin \alpha \rho\pi }{\pi} \cdot \frac{ x^{\alpha-1} }{x^{2\alpha} + 2r (\cos \alpha\rho\pi) x^\alpha + r^2}, & x>0, \\[5mm]
 \displaystyle \frac{ r \sin \alpha (1-\rho)\pi }{\pi} \cdot \frac{ |x|^{\alpha-1} }{|x|^{2\alpha} + 2r (\cos \alpha(1-\rho)\pi) |x|^\alpha + r^2} ,& x<0, 
\end{cases}
\end{equation}
see \cite{HS15}. For further information, see \cite{HS15,AH16} and the original article \cite{SW97}. 

Finally, we mention that the {\it Cauchy distribution} 
\begin{equation}
\bfc_{\beta,\gamma}(d x) = \frac{\gamma}{\pi}\cdot\frac{1}{(x-\beta)^2+\gamma^2}\Ind_{\R}(x)\,dx,\qquad\beta \in \R,\gamma>0, \qquad \bfc_{\beta,0}=\delta_\beta   
\end{equation}
plays a special role since it is a strictly $1$-stable distribution in classical, free and Boolean senses, and it satisfies 
\begin{equation}\label{special Cauchy}
\bfc_{\beta,\gamma} \ast \mu = \bfc_{\beta,\gamma} \boxplus \mu = \bfc_{\beta,\gamma} \uplus \mu
\end{equation}  
for all probability measures $\mu$. The Voiculescu transform of the Cauchy distribution is given by 
\begin{equation}
\varphi_{\bfc_{\beta,\gamma}}(z) = \beta -\ii \gamma,\qquad z\in\C^+. 
\end{equation}

\subsection{Multiplicative classical convolutions}
Let $G$ be either $(0,\infty)$ or $\T$. The multiplicative classical convolution $\mu_1\circledast\mu_2$ of Borel probability measures $\mu_1$ and $\mu_2$ on $G$ is the law of $X_1 X_2$, where $X_1$ and $X_2$ are independent, $G$-valued random variables such that $\Law(X_i)=\mu_i, i=1,2$. In fact for $G=(0,\infty)$, the multiplicative group $((0,\infty), \cdot)$ is isomorphic to the additive group $(\R,+)$ by the exponential map, and so L\'evy processes and probability measures on $(0,\infty)$ can be identified with those on $\R$. 
For the unit circle $G=\T$, such an identification is not possible since the map $x \mapsto e^{\ii x}$ from $\R$ to $\T$ is not injective. However, this map is still useful to prove limit theorems for L\'evy processes (see Sections \ref{exponential} and \ref{S7}).  

The structure of $\circledast$-ID distributions on $\T$ is well known. For simplicity let us avoid the case of vanishing mean; namely let $\ID_*(\circledast, \T)$ be the set of $\circledast$-ID distributions $\mu$ on $\T$ such that $\int_{\T} \zeta \,d\mu(\zeta)\neq0$. Any such measure has the L\'evy-Khintchine representation
\begin{equation}\label{CULK}
\int_{\T}\zeta^n \,d\mu (\zeta) = \gamma \exp\left( \int_{\T} \frac{\zeta^n -1- \ii n \Im(\zeta)}{1-\Re(\zeta)}\, d\sigma(\zeta) \right), \qquad n \in \Z, 
\end{equation}
where $\gamma \in\T$ and $\sigma$ is a finite Borel measure on $\T$. Conversely for any such pair $(\gamma,\sigma)$ there exists $\mu \in \ID_*(\circledast, \T)$ such that \eqref{CULK} holds. Note that given $\mu$ the pair $(\gamma,\sigma)$ is not unique. We call $(\gamma,\sigma)$ a {\it (multiplicative) classical generating pair} of $\mu$ and denote $\mu= \mu_\circledast^{\gamma,\sigma}$. To each generating pair $(\gamma,\sigma)$ and $t\geq0$ we can associate a probability measure $\mu_\circledast^{\gamma^t, t\sigma}$, denoted by $\mu^{\circledast t}$ if we write $\mu=\mu_\circledast^{\gamma, \sigma}$. Notice that a continuous function $t\mapsto \gamma^t$ is not uniquely defined, so we need to specify its branch. Once we choose a branch, we can associate a L\'evy process on $\T$ which has the distribution $\mu^{\circledast t}$ at time $t\geq0$. For further details see \cite{Ceb16,CG08,Par67}.

\subsection{Multiplicative free convolution on the positive real line}
For (Borel) probability measures $\mu$ and $\nu$ on $[0,\infty)$, the {\it multiplicative free convolution} $\mu \boxtimes \nu$ is the distribution of $X^{1/2}YX^{1/2}$, where $X$ and $Y$ are nonnegative free random variables with distributions $\mu$ and $\nu$, respectively. This binary operation was first introduced by Voiculescu for compactly supported probability measures \cite{Voi87}, and then by Bercovici and Voiculescu \cite{BV93} for general probability measures on $[0,\infty)$. The following presentation is based on \cite{BV93,BB05}. 

To investigate multiplicative free convolution, an important transform is the $\eta$-transform defined in \eqref{eta-f}. 
If $\mu \neq \delta_0$ is a probability measure on $[0,\infty)$, then the function $\eta_\mu$ is strictly increasing in $(-\infty, 0)$, $\eta_\mu(-0)=0$ and $\eta_\mu(-\infty)=1-1/\mu(\{0\})$ (which is $-\infty$ when $\mu(\{0\})=0$) so that 
we can define the compositional inverse map $\eta_\mu^{-1}$ and further define the {\it $\Sigma$-transform}
\begin{equation}
\Sigma_\mu(z) :=\frac{\eta_\mu^{-1}(z)}{z},\qquad  1-\frac{1}{\mu(\{0\})} < z <0.  
\end{equation}
For $\mu \neq \delta_0 \neq \nu$, the identity
\begin{equation}\label{eq sigma}
\Sigma_{\mu \boxtimes \nu}(z) = \Sigma_\mu(z) \Sigma_\nu(z) 
\end{equation}
holds in the common domain (of the form $(-\alpha, 0)$ for some $\alpha>0$) of the three $\Sigma$-transforms. 
A variant of the $\Sigma$-transform is the {\it $S$-transform}, which is related to $\Sigma$ by 
\begin{equation}
\Sigma_\mu(z) = S_\mu\!\left(\frac{z}{1-z}\right). 
\end{equation}
The $\boxtimes$-ID distributions on $[0,\infty)$ are characterized in the following way.  

\begin{thm}[Bercovici-Voiculescu \cite{BV92,BV93}]\label{BV} 
A probability measure $\mu \neq \delta_0$ on $[0,\infty)$ is $\boxtimes$-ID if and only if there exists a function $\vv_\mu$ satisfying the following: 
\begin{enumerate}[\quad\rm (1)] 
\item\label{MFID1} $\vv_\mu$ is analytic in $\C\setminus [0,\infty)$, $\vv_\mu(\overline{z})=\overline{\vv_\mu(z)}$ for $z\in\C^-$, and $\vv_\mu(\C^-) \subset \C^+ \cup \R$;  
\item\label{MFID2} $\Sigma_\mu(z)=e^{\vv_\mu(z)}$ for $z\in (1-1/\mu(\{0\}),0)$.  
\end{enumerate} 
Moreover, the conditions \eqref{MFID1} are equivalent to the Pick-Nevanlinna representation
\begin{equation}\label{eq FLK}
\vv_\mu(z)=-a z + b + \int_{[0,\infty)}\frac{1 +xz}{z-x}\,d\tau(x), 
\end{equation}
where $a \geq 0$, $b \in \R$ and $\tau$ is a non-negative finite measure on $[0,\infty)$. The triplet $(a, b,\tau)$ is unique.  
\end{thm}

\begin{rem}
If $\mu$ is $\boxtimes$-ID and $\mu\neq\delta_0$, then $\mu(\{0\})=0$ from \cite[Lemma 6.10]{BV93}. Therefore, we work only on probability measures on $(0,\infty)$ when considering $\boxtimes$-ID laws.  
\end{rem}

Given a probability measure $\mu\neq\delta_0$ on $[0,\infty)$ and $t\geq1$, there exists a unique probability measure 
$
\mu^{\boxtimes t}
$
on $[0,\infty)$ such that 
\begin{equation}
\Sigma_{\mu^{\boxtimes t}}(z) = \Sigma_\mu(z)^t
\end{equation}
on some interval $(-\alpha,0)$. The reader is referred to \cite{NS96,BB05} for further details.  

The free convolution power  
$
\mu^{\boxtimes t}
$
can be extended to arbitrary $t\geq0$ if (and only if) $\mu$ is $\boxtimes$-ID. Similarly to additive free convolution, for each $\boxtimes$-ID distribution $\mu$ on $(0,\infty)$ there exists a positive MFLP $\{X_t\}_{t\geq0}$ such that $X_0=\id$ and $\Law(X_t)=\mu^{\boxtimes t}$.

\subsection{Multiplicative free convolution on the unit circle}
The multiplicative free convolution $\mu\boxtimes\nu$ of probability measures in $\cP(\T)$ is the distribution of $UV$ when $U$ and $V$ are free unitary elements such that the laws of $U$ and $V$ are $\mu$ and $\nu$, respectively \cite{Voi87}.  
Let $\mu\in\cP(\T)$. Now, we consider $G_\mu(z)$ and $F_\mu(z)$ for $z$ outside the unit disc $\disc$, and $\eta_\mu(z) =1-zF_\mu\left(\frac{1}{z}\right)$ in the unit disc $\disc$.  
Suppose that the first moment $m_1(\mu)=\int_{\T} w\, d\mu(w)$ of $\mu$ is not zero. Then the function $\eta_\mu$ has a convergent series expansion $\eta_\mu(z)=m_1(\mu)z+ o(z)$, and so  
one can define the compositional inverse $\eta_\mu^{-1}(z)$ in a neighborhood of $0$ as a convergent series, and define 
\begin{equation}\label{sigma}
\Sigma_\mu(z) :=\frac{\eta_\mu^{-1}(z)}{z}
\end{equation}
in a neighborhood of $0$. 
Suppose that $m_1(\mu) \neq 0 \neq  m_1(\nu)$. Then the multiplicative free convolution is characterized by \cite{Voi87}
\begin{equation}
 \Sigma_{\mu \boxtimes \nu}(z) = \Sigma_\mu(z) \Sigma_\nu(z)  \label{boxtimes2}
\end{equation}
in a neighborhood of $0$. It is known that only the normalized Haar measure $\Haar$ is a $\boxtimes$-ID distribution with mean $0$. Thus we introduce the class $\ID_\ast(\boxtimes, \T):= \ID(\boxtimes, \T)\setminus\{\Haar\}$.

A probability distribution $\mu$ is a member of $\ID_\ast(\boxtimes, \T)$ if and only if $\Sigma_\mu$ can be written as \cite{BV92}
\begin{equation}\label{idm}
\Sigma_\mu(z)= \gamma^{-1}\exp\Big(\int_{\T}\frac{1 +\zeta z}{1-\zeta z}\sigma(d\zeta)\Big), \qquad z \in \disc, 
\end{equation}
where $\gamma \in \T$ and $\sigma$ is a non-negative finite measure on $\T$. The pair $(\gamma,\sigma)$ is unique and is called the {\it (multiplicative) free generating pair} of $\mu$. We denote by $\mu_\boxtimes^{\gamma,\sigma}$ the $\boxtimes$-ID distribution characterized by \eqref{idm}. 
The $\boxtimes$-infinite divisibility of $\mu$ is equivalent to the existence of a weakly continuous $\boxtimes$-convolution semigroup $\{\mu^{\boxtimes t}\}_{t \geq 0}$ with $\mu^{\boxtimes 0} = \delta_1$ and $\mu^{\boxtimes 1} = \mu$.  
This convolution semigroup can be realized as the law of a unitary MFLP, whose asymptotic behaviour at time $0$ is studied in Section \ref{S7}.

\subsection{Multiplicative Boolean convolution on the positive real line}\label{sec MBC}

There is no satisfactory definition of ``multiplicative Boolean convolution on $[0,\infty)$''.  Bercovici considered a possibility of an operation $\utimes$ defined by
\begin{equation}\label{Boolmult} \frac{\eta_\mu(z)}{z} \frac{\eta_\nu(z)}{z}=\frac{\eta_{\mu\sutimes\nu}(z)}{z},  \end{equation}
but the formula \eqref{Boolmult} does not always define a probability measure on $[0,\infty)$. In fact, Bercovici showed that the power $\mu^{\utimes n}$ does not exist for sufficiently large $n$ if $\mu \in \cP([0,\infty))$ is compactly supported and non-degenerate \cite{Ber06}. Franz also tried another definition of multiplicative Boolean convolution, which turned out to be non-associative \cite{Fra09a}. 
On the other hand, Bercovici proved that the formula
\begin{equation}\label{eq MBCP}
\frac{\eta_{\mu^{\sutimes t}}(z)}{z} = \left(\frac{\eta_{\mu}(z)}{z}\right)^t,\qquad z\in(-\infty,0),  
\end{equation}
defines a probability measure $\mu^{\utimes t}$ on $[0,\infty)$ for any $0 \leq t \leq1$ and any probability measure $\mu$ on $[0,\infty)$, and this definition works well e.g.\ in \cite{AH13}. We will investigate limit theorems for this Boolean convolution power in Section \ref{S5}.

\subsection{Multiplicative Boolean convolution on the unit circle}

The multiplicative Boolean convolution $\mu\utimes\nu$ of probability measures $\mu,\nu$ on $\T$ was defined by Franz \cite{Fra09b} as the distribution of $UV$ when $U$ and $V$ are unitary elements such that $U-1$ and $V-1$ are Boolean independent and such that $\Law(U)=\mu$ and $\Law(V)=\nu$. It is characterized in terms of the $\eta$-transforms by the formula 
\begin{equation}\label{Boolmult} \frac{\eta_\mu(z)}{z} \frac{\eta_\nu(z)}{z}=\frac{\eta_{\mu\sutimes\nu}(z)}{z}. \end{equation}
Similarly to the free case, only the normalized Haar measure is a $\utimes$-ID distribution with mean $0$, and so we set $\ID_*(\utimes,\T)=\ID(\utimes,\T) \setminus  \{\Haar\}$.  
In fact, a probability measure $\mu$ is a member of $\ID_*(\utimes,\T)$ if and only if $\eta_\mu'(0) \neq 0$ and $\eta_\mu$ does not have a zero in $\disc\setminus\{0\}$. This is also equivalent to the L\'evy--Khintchine representation 
\begin{equation}
\eta_{\mu}(z) = \gamma z \exp \left( - \int_{\T} \frac{1 + \zeta z}{1 - \zeta z} \,d\sigma(\zeta) \right), \qquad z \in \disc,  
\end{equation}
where $\gamma \in \T$ and $\sigma$ is a non-negative finite measure on $\T$. 
Similarly to the free case, $\utimes$-infinite divisibility of $\mu$ is equivalent to the existence of a weakly continuous $\utimes$-convolution semigroup $\{\mu^{\boxtimes t}\}_{t \geq 0}$ with $\mu^{\boxtimes 0} = \delta_1$ and $\mu^{\boxtimes 1} = \mu$ which can be realized as the law of a unitary MBLP. We do not analyze unitary MBLPs in this paper, but the class $\ID_*(\utimes,\T)$ is important for multiplicative free convolution; see Section \ref{exponential}.

\subsection{The wrapping map} \label{exponential}

\subsubsection{The classical case}
In the last section of this paper we will study unitary MFLPs. 
For this we will use the wrapping (or exponential) map $W\colon\cP(\R)\to \cP(\T)$ defined by  
\begin{equation}
\label{Wrapping}
d(W(\mu))(e^{-\ii x}) = \sum_{n \in \mathbb{Z}} d\mu(x + 2 \pi n).  
\end{equation}
Equivalently, the map $W\colon \mc{P}(\R) \rightarrow \mc{P}(\T)$ is the push-forward induced by the map $x\to e^{-\ii x}$. Namely, $W(\mu)$ equals $\Law(e^{-\ii X})$ when $\Law(X)=\mu$.  It is straightforward from the identity $e^{- \ii (X+Y)} = e^{-\ii X} e^{-\ii Y}$ that 
\begin{equation}\label{classical hom}
W(\mu \ast \nu) = W(\mu)\circledast W(\nu)
\end{equation}
for all probability measures $\mu$ and $\nu$ on $\R$, and hence $W$ maps $\ID(\ast, \R)$ into $\ID(\circledast, \T)$. From the computation \cite[Proposition 3.1]{Ceb16} we deduce the following formula for L\'evy--Khintchine representations.  
\begin{prop}\label{Wrap ID LK}
For $\mu_*^{\xi,\tau} \in \ID(\ast,\R)$ the measure $W(\mu_*^{\xi,\tau})$ has non zero mean, and the multiplicative classical generating pair $(\gamma, \sigma)$ of $W(\mu_*^{\xi,\tau})$ is given by 
\begin{align}
\gamma &= \exp\left[-\ii \xi - \ii\int_{\R}\left(\sin x -\frac{x}{1+x^2}\right) \frac{1+x^2}{x^2}\,d\tau(x) \right], \label{Wrapped ID gamma}
\end{align}
and 
\begin{align}
\frac{1}{1-\Re(\zeta)}\,d\sigma|_{\T\setminus\{1\}}(\zeta) &= d W\left(\frac{1+x^2}{x^2} \tau|_{\R\setminus\{0\}} \right)\left|_{\T\setminus\{1\}}\right. (\zeta), \label{Wrapped ID sigma}\\ 
\sigma(\{1\}) &= \frac{1}{2}\tau(\{0\}). \label{Wrapped ID sigma atom}
\end{align}
\end{prop}

We conclude this section by showing that the map $W|_{\ID(\ast)}$ is surjective onto $\ID_*(\circledast,\T)$. 

\begin{prop}\label{prop:wrapping-classical}
Given a $\circledast$-ID law $\mu_\circledast^{\gamma,\sigma}$ on $\T$, we define 
\begin{align}
\xi & = - \arg \gamma - \int_{\R\setminus\{0\}}\left( \sin x - \frac{x}{1+x^2} \right) \frac{1+x^2}{x^2} d\tau(x), \label{eq gamma} \\
\tau(d x) &= \frac{2}{1+x^2} \sum_{n\in\Z} (\tilde\sigma \ast \delta_{2\pi n})(d x), \label{ID wrap} 
\end{align}
where $\tilde\sigma$ is a measure on $[0,2\pi) \subset \R$ defined by $\tilde\sigma(A) = \sigma(\{e^{-\ii x}: x\in A\})$ for Borel subsets $A$, and $\arg \gamma$ is an arbitrary argument. Then $W(\mu_*^{\xi,\tau})= \mu_\circledast^{\gamma,\sigma}$. 
\end{prop}
\begin{proof} It suffices to check the three relations \eqref{Wrapped ID gamma}--\eqref{Wrapped ID sigma atom}. The first and the third ones are easy to check. For the second relation, considering the $2\pi$-periodicity of the measure $(1+x^2)\,d\tau(x)$, we obtain 
\begin{equation}
\begin{split}
d W\left(\frac{1+x^2}{x^2} d\tau(x)|_{\R\setminus\{0\}}\right)\left|_{\T\setminus\{1\}} \right. (e^{-\ii x}) 
&= \sum_{n\in\Z} \frac{1}{(x-2n \pi)^2} \left[(1+x^2) \tau(d x)\right]\left|_{(0,2\pi)} \right. \\
&= \frac{1}{2 (1-\cos x)} \left[(1+x^2) \tau(d x)\right]\left|_{(0,2\pi)} \right. \\
&= \frac{1}{1-\cos x} \,d\sigma|_{\T\setminus\{1\}}(e^{-\ii x}), 
\end{split}
\end{equation}
where we naturally identified the measure $[(1+x^2)\tau(d x)]|_{(0,2\pi)}$ with a measure on $\T \setminus \{1\}$ and used the identity 
\begin{equation}\label{series}
\sum_{n\in\Z} \frac{1}{(x-2\pi n)^2} = \frac{1}{2(1-\cos x)}. 
\end{equation}
Thus the second relation holds. 
\end{proof}

\subsubsection{The free case}

Furthermore, according to \cite{AA}, the map $W$ restricted to a subclass of probability measures provides a homomorphism from additive free/Boolean convolutions to multiplicative ones on the unit circle. Define  
\[
\mathcal{F}_{\cL} = \{F\colon \mathbb{C}^+ \rightarrow \mathbb{C}^+,~\text{analytic} \mid F(z + 2 \pi) = F(z) + 2 \pi\}
\]
and
\[
\cL = \{\mu \in \cP(\R) \ |\ F_\mu\in \mathcal{F}_{\cL}\}.
\]
An analytic function $F\colon \mathbb{C}^+ \rightarrow \mathbb{C}^+$  in $\mathcal{F}_{\cL}$ is the reciprocal Cauchy transform of some probability measure if and only if
\[
F(z) = z + f(e^{\ii z}),
\]
for some analytic transformation $f\colon \mathbb{D} \rightarrow \mathbb{C}^+$. Moreover, $\cL$ is closed under the three additive convolutions $\uplus, \boxplus$, and under Boolean additive convolution powers and free additive convolution powers whenever defined. On the other hand, for $\mu \in \cL$ and $n\in\Z$, we have
\begin{equation}\label{equal-convolutions}
\delta_{2 \pi n} \uplus \mu = \delta_{2 \pi n} \boxplus \mu = \delta_{2 \pi n} \ast \mu.
\end{equation}
Hence for $\mu,\nu \in \cL$ and a convolution $\star\in\{\ast,\boxplus,\uplus\}$, we may and do write "$\mu=\nu \mod \delta_{2\pi}$" if $\mu = \nu \star \delta_{2\pi n}$ for some $n\in\Z$. This defines an equivalence relation on $\cL$ independent of the choice of a convolution $\star$.   

It was proven in \cite{AA} that restricted to the class $\cL$ the map $W$ satisfies 
\begin{equation}\label{exp} 
\exp (\ii  F_\mu(z)) = \eta_{W(\mu)}(e^{\ii z}), \qquad z \in\C^+. 
\end{equation}
 Moreover, $W$ is weakly continuous and maps $\cL$ onto $\ID_*(\utimes,\T)$. While $(W|_\cL)^{-1}$ is not a bijection, the pre-image $(W|_\cL)^{-1}(\nu)$ of each $\nu\in \ID_*(\utimes,\T)$ is equal to the set $\{\mu \ast \delta_{2\pi n}: n \in \Z\}$, where $\mu$ is any probability measure in $(W|_\cL)^{-1}(\nu)$. The most important property is that $W|_\cL$ is a homomorphism between additive free and multiplicative free convolutions (also true for Boolean and monotone convolutions).

\begin{prop}[\cite{AA}]
\label{homomorphism}
For any $\mu_1, \mu_2 \in \cL$, we have 
\[
W(\mu_1 \boxplus \mu_2)=W(\mu_1) \boxtimes W(\mu_2).   
\]
Conversely, for any $\nu_1, \nu_2 \in \ID_*(\utimes,\T)$, we have 
\[
W^{-1}(\nu_1 \boxtimes \nu_2) = W^{-1}(\nu_1) \boxplus W^{-1}(\nu_2) \mod \delta_{2\pi}.
\]
\end{prop}

Recall that multiplicative convolution powers are in general multi-valued. This ambiguity can be naturally avoided using transformation $W$. 
\begin{prop}
\label{Prop:ID-preimage}
Let $\mu \in \cL$. Then the family of distributions $\{W(\mu^{\uplus t})\}_{t\geq0}$ defines a weakly continuous $\utimes$-convolution semigroup, which we express in the form 
\[
W(\mu^{\uplus t}) = W(\mu)^{\utimes t}. 
\]
Similarly, whenever $\mu^{\boxplus t}$ is defined, the family of distributions $\{W(\mu^{\boxplus t})\}_t$ defines a weakly continuous $\boxtimes$-convolution semigroup, which we denote by  
\[
W(\mu^{\boxplus t}) = W(\mu)^{\boxtimes t}. 
\]
Moreover, $W$ maps $\ID(\boxplus)\cap\cL$ onto $\ID_*(\boxtimes,\T)$. 
\end{prop}

The following two results are not stated in \cite{AA}, so we provide the proofs. 

\begin{prop}\label{prop shift-phi} Let $\mu \in \ID(\boxplus)$ and let $\tau$ be the finite measure in \eqref{thmBV93}. The following conditions are equivalent. 
\begin{enumerate}[\rm(1)]
\item $\mu \in \cL$. 
\item $\varphi_\mu(z+2\pi) =\varphi_\mu(z)$ for all $z \in\C^+$. 
\item\label{periodic tau} The measure $(1+x^2) \tau(d x)$ is invariant under the shifts $2\pi n$ for all $n\in\Z$. 
\end{enumerate}
\end{prop}
\begin{proof}
The equivalence between (1) and (2) follows from the definition of the class $\mathcal{L}$.  That (2) implies (3) follows from the Stieltjes inversion formula. Indeed,  letting $\rho(d x) := (1+x^2)\tau(dx)$, we have that  
\begin{equation}
\begin{split}
\rho([a,b])
&=- \frac{1}{\pi}\lim _{y\downarrow 0}\int_a^b \Im\left[\varphi_\mu(x+\ii y)\right]dx \\
&= -\frac{1}{\pi} \lim _{y\downarrow 0}\int_a^b\Im\left[\varphi_\mu(x+2\pi+\ii y) \right] dx  \\
&=\rho([a+2\pi,b+2\pi]), 
\end{split}
\end{equation}
where $-\infty  < a < b <\infty$ are continuity points of $\rho$ and its $2\pi$ shift. Conversely, assume that (3) holds true. For simplicity, assuming $\xi=0$ we obtain 
\begin{equation}
\begin{split}
\varphi_\mu(z) 
&= \int_\R \left(\frac{1}{z-x}+\frac{x}{1+x^2} \right)\rho(dx) \\
&= \int_\R \left(\frac{1}{z-x}+\frac{x+2\pi }{1+(x+2\pi)^2} \right)\rho(dx) +  \int_\R \left(\frac{x}{1+x^2}-\frac{x+2\pi }{1+(x+2\pi)^2} \right)\rho(dx) \\
&= \int_\R \left(\frac{1}{z-(x-2\pi)}+\frac{x}{1+x^2} \right)\rho(dx)\\
&= \varphi_\mu(z+2\pi),  
\end{split}
\end{equation}
where we used the fact that
\begin{equation}
\begin{split}
& \int_\R \left(\frac{x}{1+x^2}-\frac{x+2\pi }{1+(x+2\pi)^2} \right)\rho(dx) \\
 &\qquad=   \lim_{n\to\infty}\int_\R \left(\frac{x+ 2\pi n }{1+(x+2\pi n)^2}-\frac{x+2\pi(n+1)}{1+(x+2\pi(n+1))^2} \right)\rho(dx) \\
 &\qquad=0.
\end{split}
\end{equation}
Thus (3) implies (2). 
\end{proof}


In the following we notice that the L\'evy-Khintchine representation used in \cite[Eq.\ (8)]{AA} for the $\Sigma$-transform was not correct, which should be replaced with \eqref{idm} in this paper. 
\begin{prop}\label{prop:wrapping-free}
Given a $\boxtimes$-ID law $\mu_\boxtimes^{\gamma,\sigma}$ on $\T$, let $(\xi, \tau)$ be defined as in \eqref{eq gamma} and \eqref{ID wrap}. Then the pre-images of $\mu_\boxtimes^{\gamma,\sigma}$ by the map $W|_{\cL}$ are given by the family $\{\mu_{\boxplus}^{\xi +2\pi n, \tau} \}_{n\in\Z} \subset \ID(\boxplus)\cap \cL$. 
\end{prop}
\begin{proof}
The fact that $W(\mu_\boxplus^{\xi+2\pi n,\tau})=\mu_\boxtimes^{\gamma,\sigma}$ follows from \cite[Proposition 26]{AA}. Conversely, let $\mu_\boxplus^{\xi',\tau'}$ be a $\boxplus$-ID distribution in $\cL$ such that $W(\mu_\boxplus^{\xi',\tau'})=\mu_\boxtimes^{\gamma,\sigma}$. Note then that $(1+x^2)\,d\tau'(x)$ is $2\pi$-periodic by Proposition \ref{prop shift-phi} \eqref{periodic tau}. Again according to \cite[Proposition 26]{AA}, the pair $(\gamma,\sigma)$ is determined by \eqref{Wrapped ID gamma}--\eqref{Wrapped ID sigma atom} with $(\xi,\tau)$ replaced by $(\xi',\tau')$. Using \eqref{Wrapped ID sigma} and \eqref{series} shows that 
\begin{equation}
\begin{split}
\frac{1}{1-\cos x} \,d\sigma|_{\T\setminus\{1\}}(e^{-\ii x}) 
&= d W\left(\frac{1+x^2}{x^2} d\tau'(x)|_{\R\setminus\{0\}}\right) \left|_{\T\setminus\{1\}} \right. (e^{-\ii x})\\
&= \sum_{n\in\Z} \frac{1}{(x-2n \pi)^2} \left[(1+x^2) \tau'(d x)\right]\left|_{(0,2\pi)} \right. \\
&= \frac{1}{2 (1-\cos x)} \left[(1+x^2) \tau'(d x)\right]\left|_{(0,2\pi)}, \right. 
\end{split}
\end{equation}
where we naturally identified the measure $[(1+x^2)\tau'(d x)]|_{(0,2\pi)}$ with a measure on $\T \setminus \{1\}$. The same computation holds for $\tau$ instead of $\tau'$. Considering $\tau'(\{0\})= 2\sigma(\{1\}) = \tau(\{0\})$, we have $(1+x^2)\,d\tau'(x) = (1+x^2)\,d\tau'(x)$ on $[0,2\pi)$, and by periodicity, on $\R$.  This shows that $\tau'= \tau$. It is easy to show that $\xi'=\xi+ 2\pi n$ for some $n\in\Z$ from \eqref{Wrapped ID gamma}.  
\end{proof}


\subsection{Convergence of probability measures}
This section gives several facts on convergence in law of random variables. Most results are elementary. 

\begin{lem}\label{lemma limit1} Let $X,X_t, t>0$ be $\R$-valued random variables such that $X_t \tolaw X$ as $t\downarrow0$, and let $a,b\colon(0,\infty) \to \R$ be functions such that $a(t) \to \alpha \in \R, b(t)\to \beta \in\R$ as $t\downarrow0$. Then 
$$
a(t)X_t +b(t) \tolaw \alpha X +\beta \quad \text{as} \quad t\downarrow0. 
$$ 
\end{lem}
\begin{proof}  Define $\mu_t := \Law(X_t)$ and $\mu:= \Law(X)$. 
Take any $\ep>0, f \in C_b(\R)$ and take any decreasing sequence $\{t_n\}_{n\geq1}$ such that $\lim_{n\to\infty}t_n=0$. Since $\mu_{t_n} \wto \mu$, there exists $M>0$ such that $\sup_{n\in\N}\mu_{t_n}([-M,M]^c) <\ep$. Then 
\begin{equation}
\begin{split}
&\left| \bE[ f( a(t_n) X_{t_n} +b(t_n)) ] - \bE[f(\alpha X +\beta)] \right| \\
&\leq \left| \int_{\R} f(a(t_n) x + b(t_n)) \, d\mu_{t_n} - \int_{\R} f(\alpha x +\beta) \,d\mu_{t_n}\right| \\
&\qquad+  \left| \int_{\R} f(\alpha x+\beta) \,d\mu_{t_n} - \int_{\R} f(\alpha x+\beta) \, d\mu\right|\\
&\leq  2 \|f\|_\infty \mu_{t_n}([-M,M]^c) +  \int_{[-M,M]} \left| f(a(t_n) x+b(t_n))-f(\alpha x+\beta)\right| \, d\mu_{t_n}\\ 
&\qquad+\left| \int_{\R} f(\alpha x+\beta) \,d\mu_{t_n} - \int_{\R} f(\alpha x+\beta) \, d\mu\right|. 
\end{split}
\end{equation}
The first term is bounded by $2 \|f\|_\infty\ep$, the second integral is bounded by $\ep$ for large $n \in\N$ by the uniform continuity of $f$ on finite intervals, and the third integral tends to 0 as $n\to\infty$ by the weak convergence $\mu_{t_n}\wto\mu$. Thus we have shown that $\bE[ f( a(t_n) X_{t_n} +b(t_n)) ] \to \bE[f(\alpha X +\beta)]$ for any sequence $t_n\downarrow0$, which implies that $\bE[ f( a(t) X_{t} +b(t)) ] \to \bE[f(\alpha X +\beta)]$ as $t\downarrow0$. 
\end{proof} 

This lemma can be expressed in the multiplicative form. 
\begin{lem}\label{lemma limit2}
 Let $X, X_t$ be $(0,\infty)$-valued random variables for $t>0$ such that $X_t \tolaw X$ as $t\downarrow0$, and let 
$a\colon (0,\infty) \to \R$ and $b\colon (0,\infty)\to (0,\infty)$ be functions such that $a(t) \to \alpha \in \R, b(t)\to \beta \in(0,\infty)$ as $t\downarrow0$. Then 
$$
b(t)(X_t)^{a(t)} \tolaw \beta X^\alpha \quad \text{as} \quad t\downarrow0. 
$$
\end{lem}
\begin{proof}
The group isomorphism $\exp\colon \R \to (0,\infty)$ changes dilation to power and shift to dilation, and hence  Lemma \ref{lemma limit1} is available. 
\end{proof}

If we assume that the limit distribution is non-degenerate (i.e.\ not a point mass) and $a(t)>0$ then we can show the converse result of Lemma \ref{lemma limit1}. 
\begin{lem}\label{GK} Let $X, X_t$ be $\R$-valued random variables and let $a(t)>0, b(t)\in\R$ for $t >0$. Assume that $X_t \tolaw X$ as $t \downarrow0$ and $X$ is non-degenerate. Then $a(t) X_t +b(t)$ converges in law to some non-degenerate random variable $Y$ if and only if $a(t)$ and  $b(t)$ respectively converge to some $\alpha\in(0,\infty)$ and $\beta\in\R$ as $t\downarrow0$. Moreover, $Y\eqlaw \alpha X+\beta$.   
\end{lem}
\begin{rem} By the transform $t\mapsto 1/t$, the same statement holds for the limit $t\to\infty$.  
\end{rem}
\begin{proof} If $a(t) \to \alpha \in(0,\infty)$ and $b(t) \to \beta \in\R$ then $a(t) X_t + b(t) \tolaw \alpha X +\beta$ by Lemma \ref{lemma limit1}. Conversely, 
suppose that $a(t) X_t +b(t)$ converges in law to some non-degenerate random variable $Y$. Take a sequence $t_n \downarrow 0$ and consider the discretized random variables $a(t_n) X_{t_n}+b(t_n)$. Then \cite[pp.\ 40, Theorem 1]{GK54} implies that there exist $\alpha>0$ and $\beta\in\R$ such that $Y \eqlaw \alpha X+\beta$. Then taking $\beta_n =1, \alpha_n=0$ and replacing $b_n$ and $a_n$ respectively with $\frac{\alpha}{a(t_n)}$ and $\frac{\beta - b(t_n)}{a(t_n)}$ in \cite[pp.\ 42, Theorem 2]{GK54}, we obtain the convergence $a(t_n) \to \alpha$ and $b(t_n) \to \beta$. Since the sequence $\{t_n\}_{n\in\N}$ was arbitrary, the conclusion follows. 
\end{proof}

The multiplicative version follows by the isomorphism $\exp\colon \R \to (0,\infty)$.  

\begin{lem}\label{GK2} Let $X, X_t$ be $(0,\infty)$-valued random variables and let $a(t)>0, b(t)>0\in\R$ for $t >0$. Assume that $X_t \tolaw X$ as $t \downarrow0$ and $X$ is non-degenerate. Then $b(t) (X_t)^{a(t)}$ converges in law to some non-degenerate random variable $Y$ if and only if $a(t)$ and  $b(t)$ respectively converge to some $\alpha\in(0,\infty)$ and $\beta\in (0,\infty)$ as $t\downarrow0$. Moreover, $Y\eqlaw \beta X^\alpha $.   
\end{lem}

We give a sufficient condition for weak convergence in terms of the local uniform convergence of the absolutely continuous part. 
\begin{lem}\label{convergence1} Let $B$ be an open subset of $\R$. 
Let $\{\mu_t\}_{t>0}$ be a family of Borel probability measures on $\R$ and let $p\colon B \to [0,\infty)$ be a Borel measurable function such that $\int_{B} p(x)\,dx=1$. 
Suppose that for any compact subset $K \subset B$ there exists  $\delta>0$ such that $\mu_t$ is Lebesgue absolutely continuous on $K$ for any $0 < t <\delta$, and 
$$
\lim_{t\downarrow0}\sup_{x\in K} \left| \frac{d\mu_t}{dx}(x) -p(x) \right| =0. 
$$ 
Then $\mu_t$ converges weakly to the probability measure $p(x)\Ind_B(x) \,dx$. 
\end{lem}
\begin{proof} Denote by $\text{Leb}$ the Lebesgue measure on $\R$. Take $f \in C_b(\R)$ and $\ep>0$. Since $p(x)\Ind_B(x)\, dx$ is a Radon measure on $\R$, there exists a compact subset $K$ of $B$ such that $\int_{K} p(x)\,dx >1-\ep$. Take $\delta>0$ such that $\mu_t$ is absolutely continuous w.r.t.\ $\text{Leb}|_{K}$ for all $t\in(0,\delta)$, and also satisfying that 
$$
\sup_{x\in K} \left| \frac{d\mu_t}{dx}(x) -p(x) \right| <\frac{\ep}{\max\{\text{Leb}(K),1\}} 
$$
for $t \in(0,\delta)$. Then $\mu_t(K)  \geq \int_{K} p(x)\,dx -\ep > 1-2\ep$. Then 
\begin{equation*}
\begin{split}
&\left| \int_\R f(x)\,\mu_t(dx) -\int_\R f(x)p(x)\Ind_B(x)\,dx \right|  \\
&\qquad\leq 3\ep \|f\|_\infty + \left| \int_{K} f(x)\frac{d\mu_t}{dx}(x)\,dx -\int_{K} f(x)p(x)\,dx \right| \\
&\qquad\leq 4\ep \|f\|_\infty
\end{split}
\end{equation*}
for $t\in(0,\delta)$. 
\end{proof}



\section{Organizing easy or known limit theorems} \label{S3}
This section summarizes known results and results that follow readily from known results. 

\subsection{Additive L\'evy processes at large and small time}\label{sec ALP}

Let $\{X_t\}_{t\geq0}$ be an AFLP such that $X_0=0$. We discuss the convergence of the process 
\begin{equation}\label{conv ALP}
a(t)X_t +b(t)  \quad \text{as} \quad t\to\infty \quad \text{or} \quad t\downarrow0,
\end{equation}
where $a\colon(0,\infty) \to (0,\infty)$ and $b\colon (0,\infty) \to \R$ are some functions. Alternatively, the above problem reads the weak convergence of 
\begin{equation}
\D_{a(t)}(\mu^{\boxplus t}) \boxplus \delta_{b(t)}. 
\end{equation}
Actually this problem can be solved by Bercovici-Pata bijection, and the result has a complete correspondence to a classical result.  

\begin{rem} If \eqref{conv ALP} converges in law to a non-degenerate $\R$-valued random variable $Y$, then Lemma \ref{GK} shows that the choice of functions $a, b$ are essentially unique: For other functions $\widetilde{a}, \widetilde{b}$, 
\begin{equation}\label{uniqueness}
\widetilde{a}(t)X_t +\widetilde{b}(t) = \frac{\widetilde{a}(t)}{a(t)}[a(t)X_t +b(t)] +\widetilde{b}(t) -  \frac{\widetilde{a}(t) b(t)}{a(t)}
\end{equation}
converges in law to a non-degenerate $\R$-valued random variable $\widetilde{Y}$ if and only if there exist $\alpha >0, \beta \in\R$ such that $\frac{\widetilde{a}(t)}{a(t)} \to \alpha$ and $\widetilde{b}(t) -  \frac{\widetilde{a}(t) b(t)}{a(t)} \to \beta$, and in this case $\widetilde{Y}\eqlaw \alpha Y + \beta$. Thus it suffices to find one specific pair of functions $(a(t),b(t))$ for which the distribution of \eqref{conv ALP} converges. 
\end{rem}

First we establish that the limit distribution of \eqref{conv ALP}, if exists, must be free stable. This fact follows from \cite[Theorem 2.3]{MM08} and the Bercovici--Pata bijection, but we give a direct simple proof which is valid for classical and Boolean cases as well.  
\begin{prop}\label{prop stable}
Let $\{\mu^{\ast t}\}_{t\geq0}$ be a weakly continuous $\ast$-convolution semigroup such that $\mu^{\ast 0}=\delta_0$. If there exist functions $a\colon (0,\infty) \to (0,\infty)$ and $b\colon(0,\infty)\to\R$  such that $\D_{a(t)}(\mu^{\ast t})\ast\delta_{b(t)}$ converges weakly to a non-degenerate distribution $\nu$ as $t \downarrow0$ or as $t\to\infty$, then $\nu$ is stable. If $b(t)\equiv0$ then $\nu$ is strictly stable. An analogous statement holds for weakly continuous $\boxplus$- and $\uplus$-convolution semigroups. 
\end{prop}
\begin{proof} We only focus on the limit $t\downarrow0$ since the other case is proved in the same way. Instead of distributions we use stochastic processes. Let $\{X_t\}_{t\geq0}$ be an ACLP that has the distribution $\mu^{\ast t}$ at time $t\geq0$, and let $Y$ be a non-constant random variable such that $\Law(Y)=\nu$. Take i.i.d.\ copies $(Y_i)_{i=1}^\infty$  of $Y$. 
The following identity holds true for each $n \in\N$:  
\begin{equation}\label{identity}
\begin{split}
&\frac{a(t)}{a(nt)}\left\{a(nt)X_{nt} +b(nt) \right\} -b(nt)\frac{a(t)}{a(nt)}+ n b(t) \\
&\quad=\left\{a(t)X_t +b(t)\right\} + \left\{a(t) (X_{2t} -X_t) +b(t)\right\} + \cdots + \left\{a(t)(X_{nt}-X_{(n-1)t}) +b(t)\right\}. 
\end{split}
\end{equation}
Since $a(nt)X_{nt} +b(nt)$ converge in law to $Y$ as $t\downarrow0$ and since the right hand side of (\ref{identity})  converge in law to $Y_1+\cdots+Y_n$, 
it holds true from Lemma \ref{GK} that $\frac{a(t)}{a(nt)}$ converge to some $\alpha_n\in [0,\infty)$ and $-b(nt)\frac{a(t)}{a(nt)}+ n b(t)$ converge to  some $\beta_n\in\R$  as $t \to 0$, and also  
$$
Y_1+\cdots+Y_n\eqlaw \alpha_nY+\beta_n. 
$$
Since $Y$ is not a constant, we must have $\alpha_n\in(0,\infty)$. This implies that $Y$ is stable, see \cite[p.14, Equation I.24]{Zol86}. 
 If $b(t)\equiv0$ then $\beta_n=0$ and so $Y$ is strictly stable. The proof for the free case is similar. 
\end{proof}

\begin{thm}\label{thm ALP} Let $\mu$ be a $\ast$-ID distribution. Let $a\colon(0, \infty)\to (0,\infty)$ and $b\colon(0,\infty) \to\R$ be functions, and $\nu$ be a stable distribution or a delta measure. Then the following are equivalent. 
\begin{enumerate}[\quad\rm(1)]
\item\label{Conv1} $\D_{a(t)}(\mu^{\ast t})\ast \delta_{b(t)} \wto \nu$ as $t \to \infty$ (resp.\ $t\downarrow 0$). 
\item\label{Conv2} $\D_{a(t)}(\Lambda(\mu)^{\boxplus t})\boxplus \delta_{b(t)} \wto \Lambda(\nu)$ as $t\to\infty$ (resp.\ $t\downarrow 0$). 
\item\label{Conv3} $\D_{a(t)}(\Lambda_B(\mu)^{\uplus t})\uplus \delta_{b(t)} \wto \Lambda_B(\nu)$ as $t\to\infty$ (resp.\ $t\downarrow 0$). 
\end{enumerate} 
\end{thm}
\begin{rem}\label{rem:monotone}
The Bercovici-Pata-type bijection $\Lambda_M$ from $\ast$-ID distributions onto monotonic ID distributions is defined in \cite{Has10}, but only $\Lambda_M$ is known to be continuous. Therefore we cannot prove a monotone analogue of Theorem \ref{thm ALP}. Establishing the continuity of $\Lambda_M^{-1}$ is an open problem. 
\end{rem}
\begin{proof} For the equivalence between \eqref{Conv1} and \eqref{Conv2} we only have to use the distributional identities $\Lambda(\D_{a(t)}(\mu^{\ast t})\ast \delta_{b(t)}) = \D_{a(t)}(\Lambda(\mu)^{\boxplus t})\boxplus \delta_{b(t)}$ and  the fact that the Bercovici-Pata bijection $\Lambda$ is a homeomorphism. The equivalence between \eqref{Conv1} and \eqref{Conv3} is proved similarly. 
\end{proof}

Let $\star$ denote any one of $\ast$ and $\boxplus$. In the present context, the {\it $\star$-domain of attraction of a probability measure $\nu$ on $\R$ at large time (resp.\ small time)} is the set of all $\star$-ID distributions $\mu$ on $\R$ such that $\D_{a(t)}(\mu^{\star t}) \star \delta_{b(t)} \wto \nu$ as $t\to\infty$ (resp.\ $t\downarrow0$) for some functions $a,b$ from $(0,\infty)$ into itself. This set being denoted by $\fD^\infty_\star(\nu)$ (resp.\ $\fD^0_\star(\nu)$), the above result shows that
$$
\fD^\infty_\ast(\nu) = \fD^\infty_\boxplus(\Lambda(\nu)) \qquad \text{and} \qquad \fD^0_\ast(\nu) = \fD^0_\boxplus(\Lambda(\nu)),  
$$
and they are nonempty if and only if $\nu$ is stable or degenerate.

A complete description of the domains of attraction of stable distributions is known in \cite[Theorem 2.3]{MM08} at small time and in \cite[Theorem 3]{MM09} at large time. For later use we quote the result for small time in a slightly different form which can be deduced from the proof. 
\begin{thm}\label{DA at 0} Let $\mu$ be a $\ast$-ID distribution with classical generating pair $(\xi,\tau)$. Define 
\begin{equation*}
\begin{split}
V(x) &= \int_{|y| \leq x} (1+y^2) \,d\tau(y), \qquad \overline{\Pi}^-(x) = \int_{-\infty}^{-x} \frac{1+y^2}{y^2}\,d \tau(y), \\
\overline{\Pi}^+(x) &= \int_x^{\infty} \frac{1+y^2}{y^2}\,d \tau(y), \qquad \overline{\Pi}(x) = \overline{\Pi}^+(x)+\overline{\Pi}^-(x), \qquad x>0.  
\end{split}
\end{equation*}
\begin{enumerate}[\rm(1)] 
\item $\mu \in \fD^0_\ast(\bfs_{2,1/2})$ if and only if the function $V$ is slowly varying as $x\downarrow0$. 

\item Let $(\alpha,\rho) \in\Ad, \alpha \neq2$. Then $\mu \in \fD^0_\ast(\bfs_{\alpha,\rho})$ if and only if the function $\overline{\Pi}$ is regularly varying with index $-\alpha$ as $x\downarrow0$, and 
$$
\lim_{x\downarrow0} \frac{\overline{\Pi}^+(x)}{\overline{\Pi}(x)} = 
\begin{cases}
\frac{1}{2} \left(1 + \frac{\tan(\rho-\frac{1}{2})\alpha\pi}{\tan\frac{\alpha \pi}{2}}\right), &\alpha \neq 1, \\
\rho, & \alpha =1. 
\end{cases}
$$ 
\end{enumerate}
\end{thm}
\begin{proof}
Most of the statements can be inferred from \cite[Theorem 2.3]{MM08} and its proof. 
We only mention that the finite measure $\tau_{\alpha,\rho}$ that appears in the L\'evy--Khintchine representation of the stable distribution $\bfs_{\alpha,\rho}$ is given by 
\begin{equation}
(1+x^2)\tau_{\alpha,\rho}(d x) = 
\begin{cases} 
\frac{\sin(\alpha(1-\rho)\pi)}{\pi} |x|^{1-\alpha}\Ind_{(-\infty,0)}(x)\,dx +  \frac{\sin(\alpha\rho\pi)}{\pi} |x|^{1-\alpha}\Ind_{(0,\infty)}(x)\,dx, & \alpha\neq1, \\
(1-\rho)\Ind_{(-\infty,0)}(x)\, dx +  \rho \Ind_{(0,\infty)}(x)\,dx, &\alpha=1, 
\end{cases}
\end{equation}
and hence the proof of  \cite[Theorem 2.3]{MM08} shows that 
\begin{equation}
\lim_{x\downarrow0} \frac{\overline{\Pi}^+(x)}{\overline{\Pi}(x)} = 
\begin{cases}
\frac{\sin(\alpha\rho\pi)}{\sin(\alpha\rho\pi)+\sin(\alpha(1-\rho)\pi)}, & \alpha\neq1, \\
\rho, & \alpha=1. 
\end{cases}
\end{equation}
Elementary formulas for trigonometric functions show that 
\begin{equation}
\frac{\sin(\alpha\rho\pi)}{\sin(\alpha\rho\pi)+\sin(\alpha(1-\rho)\pi)} = \frac{1}{2} \left(1 + \frac{\tan(\rho-\frac{1}{2})\alpha\pi}{\tan\frac{\alpha \pi}{2}}\right),\qquad \alpha \neq 1
\end{equation}
Note that the last expression is an increasing function of $\rho$ when $0<\alpha<1$ and decreasing function when $1<\alpha<2$. 
\end{proof}

\subsection{Positive multiplicative free L\'evy processes at large time}\label{sec:Haagerup-Moeller}
We consider convergence in law of the process
\begin{equation}\label{large time}
b(t) (X_t)^{a(t)}, \qquad t\to\infty, 
\end{equation}
where $a,b\colon (0,\infty) \to (0,\infty)$ are functions and $\{X_t\}_{t\geq0}$ is a positive MFLP such that $X_0$ is an identity operator. In terms of the marginal law $\mu:= \Law(X_1)$, the above problem is written in the form
\begin{equation}\label{large time2}
\D_{b(t)} \!\left((\mu^{\boxtimes t})^{a(t)}\right).  
\end{equation}
In fact, the convergence of \eqref{large time} follows from the work of Tucci \cite{T10} and Haagerup-M\"{o}ller \cite{HM} and we can take the functions $a(t)= 1/t$ and $b(t) \equiv1$.  
Tucci \cite{T10} initiated the study of law of large numbers for multiplicative free convolution for compactly supported probability measures, and then Haagerup and M\"{o}ller \cite{HM} gave a proof of the general case. The following statement is formulated in the setting of continuous time which can be proved without changing the proof. 
\begin{thm}\label{THMthm}
Let $\mu$ be a probability measure on $[0,\infty)$. Then 
$$
(\mu^{\boxtimes t})^{1/t} \wto \Phi(\mu) \quad \text{as} \quad t\to\infty, 
$$
 where $\Phi(\mu)$ is characterized by 
$$
\Phi(\mu)\left(\left[0,\frac{1}{S_\mu(x-1)}\right]\right) = x,\quad x\in(\mu(\{0\}),1). 
$$
It holds that $\Phi(\mu)(\{0\})=\mu(\{0\})$ and the support of $\Phi(\mu)$ is the closure of the interval 
$$
\left( \left(\int_{[0,\infty)} \frac{1}{x}\,\mu(dx)\right)^{-1}, \int_{[0,\infty)}x\,\mu(dx)\right). 
$$ 
Moreover, $\Phi(\mu)$ is non-degenerate if and only if $\mu$ is non-degenerate. 
\end{thm}
Thus the problem of finding limit distributions of MFLPs at large time has been settled by Theorem \ref{THMthm}; we only need to restrict the initial measure $\mu$ to $\boxtimes$-ID laws on $(0,\infty)$. 

Let $\fD^\infty_\boxtimes(\nu)$ denote the $\boxtimes$-domain of attraction of $\nu$ at large time, i.e.\ the set of all $\boxtimes$-ID distributions $\mu$ on $(0,\infty)$ such that \eqref{large time2} converges to $\nu$ for some functions $a,b\colon (0,\infty)\to(0,\infty)$. 

The map $\mu \mapsto \Phi(\mu)$ is injective since the map $\mu \mapsto S_\mu$ is injective. This fact and arguments similar to the paragraph around \eqref{uniqueness} with Lemma \ref{GK2} completely determine the $\boxtimes$-domain of attraction of $\Phi(\mu)$ for a non-degenerate $\boxtimes$-ID law $\mu$ on $(0,\infty)$: 
$$
\fD_\boxtimes^\infty(\Phi(\mu))=\{\D_\beta (\mu^\alpha): \alpha,\beta >0 \}. 
$$  
Thus the limit distributions are not universal since each domain of attraction is small. 

By contrast, in classical probability, the limiting distributions of positive MCLPs at large (and small) time are universal.
\begin{prop}\label{multiplicativeclassical}
Let $\{X_t\}_{t\geq0}$ be a MCLP on $(0,\infty)$ such that $X_0=1$. 
If there exist functions $a, b\colon (0,\infty) \to (0,\infty)$ such that 
 $b(t) X_t^{a(t)}$ converges in law  to a non-constant positive random variable $Y$ as $t \to\infty$ or $t\downarrow0$,  then $Y$ is log stable, i.e.\ $\log Y$ is stable.  
\end{prop}
\begin{proof}
 This follows from the additive case (Proposition \ref{prop stable}) applied to the ACLP $Z_t = \log(X_t)$. 
\end{proof}

Note that the Boolean analogue of \eqref{large time} cannot be formulated, as  Bercovici showed that a reasonable distribution $\mu^{\utimes n}$ does not exist for sufficiently large $n$ if $\mu$ is compactly supported and non-degenerate \cite{Ber06}. The monotone version of \eqref{large time} can be formulated but is not discussed in this paper.

\section{Positive multiplicative free L\'evy processes at small time} \label{sec PMFLP}
We consider the limit theorem of the type \eqref{large time}, but for small time. In terms of probability measures, the problem is convergence  
\begin{equation}\label{small time}
\D_{b(t)}(\mu^{\boxtimes t})^{a(t)}, \qquad t\downarrow0, 
\end{equation}
where $a,b\colon (0,\infty) \to (0,\infty)$ are functions and $\mu$ is a $\boxtimes$-ID distribution on $(0,\infty)$. 
In the case of large time, recall that we can always take $a(t)=1/t$ and $b(t)\equiv 1$. However, such is not anymore true at small time. 

In classical probability, the possible limit distributions are only log stable distributions and degenerate distributions; see Proposition \ref{multiplicativeclassical}. 
Our results for free case are similar to this classical case; we find log free stable distributions as the limit distribution of \eqref{small time}.

\subsection{Log Cauchy distribution}
In this section we present a limit theorem \eqref{small time} when the functions $a(t)=1/t$ and $b(t)\equiv 1$ can be taken. 
Let $\bfC_{\beta,\gamma}$ be a random variable following the Cauchy distribution $\bfc_{\beta,\gamma}$. 
The law $\Law(e^{\bfC_{\beta,\gamma}})$ is called the \emph{log Cauchy distribution} whose probability density is given by 
$$
\frac{\gamma}{\pi x}\cdot\frac{1}{(\log x -\beta)^2+\gamma^2}\Ind_{(0,\infty)}(x). 
$$

The main theorem here is the convergence to the log Cauchy distribution.  

\begin{thm}\label{LCF} Let $\mu$ be a $\boxtimes$-ID probability measure on $(0,\infty)$. Assume that the analytic function $\vv_\mu$ in \eqref{eq FLK} extends to a continuous function in $(\ii\C^+) \cup \C^- \cup I$ where $I$ is an open interval containing $1$, and assume that $-\beta+\ii\gamma:= \vv_\mu(1) \in \C^+$. Then for any compact set $K \subset (0,\infty)$, the measure $(\mu^{\boxtimes t})^{1/t}$ is Lebesgue absolutely continuous on $K$ for small $t>0$, and the convergence 
\[
\frac{d (\mu^{\boxtimes t})^{1/t}}{d x} \to \frac{\gamma}{\pi x [(\log x - \beta)^2 + \gamma^2]} \quad \text{as}\quad t \downarrow0
\] 
holds uniformly on $K$. In particular, $(\mu^{\boxtimes t})^{1/t}$ converges to $\Law(e^{\mathbf{C}_{\beta,\gamma}})$ weakly. 
\end{thm}
\begin{rem}  
The assumption on $\vv_\mu$ is guaranteed if the generating measure $\tau_\mu$ in \eqref{eq FLK} is Lebesgue absolutely continuous on $I$ and $d\tau_\mu/d x$ is locally H\"older continuous and strictly positive on $I$. See Example \ref{exa Holder} for further details. 
\end{rem}
We reduce the problem to the Boolean case, and the proof is postponed to Section \ref{sec6}. The idea is the following. Suppose that we find a probability measure $\nu$ such that $\mu =(\nu^{\boxtimes 2})^{\sutimes \frac{1}{2}}$. Then using a commutation relation in \cite{AH13} we obtain 
\begin{equation}\label{key eq2}
\mu^{\boxtimes t}= \left[ (\nu^{\boxtimes (1+t)})^{\sutimes \frac{1}{1+t}}\right]^{\sutimes t}. 
\end{equation}
The measure $ (\nu^{\boxtimes (1+t)})^{\sutimes \frac{1}{1+t}}$ is approximately $\nu$ when $t \downarrow0$, and so the study of $\D_{b(t)}(\mu^{\boxtimes t})^{a(t)}$ reduces to the study of $\D_{b(t)}(\nu^{\sutimes t})^{a(t)}$ which is easier. 
 The relation between $\mu$ and $\nu$ is that $\mu$ is the image of $\nu$ by the multiplicative Bercovici--Pata map, which is not a bijection. Therefore, for some $\mu \in \ID(\boxtimes)$, we cannot find such a pre-image $\nu$. However, we do not need a ``probability measure'' $\nu$, but only need its $\eta$-transform. This idea will be made more precise in Section \ref{sec6}. The equation \eqref{key eq2} will then be generalized to \eqref{key eq}.

\begin{exa}\label{LC BS} The positive Boolean $\alpha$-stable law $(0<\alpha<1)$ has the $\Sigma$-transform 
$
\Sigma_{\bfb_\alpha}(z) = (-z)^{\frac{1-\alpha}{\alpha}}, 
$
and so
$$
\vv_{\bfb_\alpha}(z) = \frac{1-\alpha}{\alpha}\log (-z). 
$$
This implies that 
 $$
 \lim_{y\uparrow0}\vv_{\bfb_\alpha}(1+\ii y) = \ii \frac{(1-\alpha)\pi}{\alpha}, 
 $$
 and hence we get the convergence 
 $$
 \frac{d (\bfb_\alpha^{\boxtimes t})^{1/t}}{dx} \to \frac{1}{\pi x} \cdot\frac{\gamma}{(\log x)^2 + \gamma^2}
 $$
 uniformly on each compact set of $(0,\infty)$, where $\gamma=(1-\alpha)\pi/\alpha$. 
\end{exa}

\subsection{Dykema-Haagerup distribution}
In this section we find a limit distribution of \eqref{small time} which is not a log Cauchy distribution but still a log free stable distribution with index $1$. 
Dykema and Haagerup \cite{DH04a} investigated the $N \times N$ strictly upper triangular random matrix
\begin{equation*}
T_N:= 
\begin{pmatrix} 
0 & t_{12} & t_{13} & \cdots & t_{1, N-1} & t_{1 N}\\
0 & 0 & t_{23} & \cdots & t_{2, N-1} & t_{2 N}\\
0 & 0 & 0 & \cdots & t_{3, N-1} & t_{3N}\\
\vdots & \vdots  &&\ddots & & \vdots  \\
0 & 0 & 0 & \cdots & 0 & t_{N-1, N}\\
0&0&0&\cdots &0 & 0
\end{pmatrix}
\end{equation*}
where the entries $\{t_{i j}\}_{1 \leq i < j \leq N}$ are independent complex Gaussian with mean 0 and variance $1/n$. 
The showed that $(T_N, \bE\otimes \frac{1}{N}\Tr_N)$ converges in $\ast$-moments to some $(T,\tau)$, where $\tau$ is a trace. The operator $T$ is called the DT-operator. 
They conjectured that 
$$
\tau\!\left[((T^\ast)^kT^k)^n\right]=\frac{n^{kn}}{(1+kn)!},\qquad k,n \in \N, 
$$ 
which was proved by Dykema and Haagerup for $k=1$ and then proved by \'Sniady \cite{S03} in full generality. 
Cheliotis showed that the empirical eigenvalue distribution of $T_N^* T_N$ converges weakly almost surely \cite[Theorem 1]{Che}. 
A similar but different random matrix model was found by Basu et al.\ \cite[Theorem 3.1]{BBGH}. 

Generalizing natural numbers $k$ to positive real numbers, we introduce a probability measure $\DH_{r}$ $(r\geq0)$ whose moments are given by 
\begin{equation}\label{DH}
\frac{n^{r n}}{\Gamma(2+r n)},\quad n=0, 1,2,3,\ldots, \quad r\geq0, 
\end{equation}
with the convention $0^0=1$. More generally, the Mellin transform is given by 
\begin{equation}\label{DH2}
\int_{[0,\infty)} x^\gamma\,\DH_r(dx) = \frac{\gamma^{r\gamma}}{\Gamma(2+r\gamma)},\quad r,\gamma>0. 
\end{equation} 
The existence of such a probability measure is guaranteed by Theorem \ref{free Bessel} in this section because its proof implies the positive definiteness of the sequence \eqref{DH}. 
We call $\DH_{r}$ the \emph{Dykema-Haagerup distribution}. 
It can be easily shown that 
\begin{equation}\label{powerrelation}
(\DH_r)^a = \D_{a^{ar}}(\DH_{ar}),\quad a,r\geq0. 
\end{equation}
The probability distribution $\DH_1$ is the spectral distribution of the DT operator $T^* T$, and it is Lebesgue absolutely continuous and is supported on $[0,e]$ \cite[Theorem 8.9]{DH04a}. Hence, $\DH_r = \D_{r^{-r}}(\DH_1^r)$ is supported on $[0,r^{-r}e^r]$ for $r>0$. 
The $R$-transform of $\DH_1$ is explicitly computed in \cite[Theorem 8.7]{DH04a}, which is not used in this paper. 

The Dykema-Haagerup distribution is in fact a log free stable distribution, which seems unknown in the literature.

\begin{prop}\label{prop DH} Let $\bfF_1$ be a random variable following the free stable law $\bff_1$. Then 
$$
\DH_1 = \Law(e^{\bfF_1}). 
$$
\end{prop}
\begin{proof}
Dykema and Haagerup obtained an implicit expression of the density $p(x)$ of $\DH_1$ in \cite[Theorem 8.9]{DH04a}: 
\begin{equation}
p\!\left( \frac{\sin \theta}{\theta}e^{\theta \cot \theta}\right) =  \frac{\sin\theta}{\pi} e^{-\theta \cot \theta},\qquad \theta \in(0,\pi). 
\end{equation}
On the other hand Biane obtained an implicit expression of the density $q(x)$ of $\bff_1$ in \cite[Proposition A1.3]{BP99}: 
\begin{equation}
q\!\left(\theta \cot\theta + \log \frac{\sin\theta}{\theta} \right) = \frac{\sin^2\theta}{\pi\theta},\qquad \theta \in (0,\pi). 
\end{equation}
The density $r(x)$ of $\Law(e^{\bfF_1})$ is given by $r(x) = x^{-1}q(\log x)$. Define 
\begin{equation}\label{x-theta}
x=e^{\theta \cot\theta} \frac{\sin\theta}{\theta}. 
\end{equation}
Then we obtain 
\begin{equation}\label{}
\begin{split}
r(x) = x^{-1}q(\log x) =  \frac{\sin^2\theta}{\pi x\theta} = \frac{\sin\theta}{\pi} e^{-\theta \cot \theta} = p(x), 
\end{split}
\end{equation} 
where we used the identity \eqref{x-theta} in the third equality. 
\end{proof}

An interesting observation here is that the free $1$-stable law has a random matrix model.  
\begin{cor}
The eigenvalue distribution of the random matrix $\log(T_N^\ast T_N)$ weakly converges to $\bff_1$ almost surely as $N\to \infty$. 
\end{cor}
We know that the semicircle law $\bff_{2,1/2}$ has the random matrix model $T_N + T_N^*$. Considering these facts, the following question comes up. 
\begin{prob} Find a random matrix model whose eigenvalue distribution converges to another free stable distribution. 
\end{prob}

In this section, we show that the Dykema-Haagerup distribution appears in the limit theorem \eqref{small time}, when we take the initial distribution $\mu$ to be the free Bessel law \cite{BBCC11}. 
Suppose that $r,s \geq0$ and $\max\{r,s\}\geq1$. The free Bessel law is defined by 
\begin{equation}
\bm\pi(r,s)
=
\begin{cases}
\bm\pi^{\boxtimes (r-1)}\boxtimes \bm\pi^{\boxplus s}, & r \geq1, s \geq0, \\
\left( (1-s)\delta_0 + s \delta_1\right) \boxtimes \bm\pi^{\boxtimes r}, & r \geq0, 0 \leq s \leq 1. 
\end{cases}
\end{equation}
The two definitions are compatible in the common domain $r \geq1$ and $0 \leq s \leq 1$. If $s\neq0$ then the $\Sigma$-transform is 
\begin{equation}\label{FBsigma}
\Sigma_{\bm\pi(r,s)}(z) = \frac{(1-z)^r}{(1-s)z +s}, 
\end{equation}
which holds for $z \in (-\infty, 0)$ if $s\geq1$ and $z \in (-s/(1-s),0)$ if $0 < s <1$. Note that $\Sigma$-transform is not defined for $\delta_0$, so the formula \eqref{FBsigma} fails for $s=0$. 

Before studying the limit theorem, we need to clarify when the free Bessel law is $\boxtimes$-ID. 

\begin{prop} Let $r,s \geq0$ and $\max\{r,s\}\geq1$. The free Bessel law $\bm\pi(r,s)$ is $\boxtimes$-ID if and only if either (a) $s=0$ and $s=1$,  or (b) $r\geq1$ and $s>1$. 
\end{prop}
\begin{proof}
Since $\bm\pi(r,0)=\delta_0$ and $\bm\pi(r,1)=\bm\pi^{\boxtimes r}$ are both $\boxtimes$-ID, we may assume that $s \neq 0,1$. 
If $r,s \geq1$, then both $\bm\pi^{\boxtimes (r-1)}$ and $\bm\pi^{\boxplus s}$ are $\boxtimes$-ID by \cite[Example 5.5]{AH13}. Hence their multiplicative free convolution is $\boxtimes$-ID as well. 
Conversely, suppose that $0\leq r<1$ or $0 < s <1$. 
The formula \eqref{FBsigma} yields 
\begin{equation}
\log \Sigma_{\bm\pi(r,s)}(z) = r \log (1-z) - \log \left( (1-s)z+s\right). 
\end{equation}
If $0<s<1$ then $\Im(\log \Sigma_{\bm\pi(r,s)}(z))$ is not analytic at $z= - s/(1-s)<0$, which implies that $\bm\pi(r,s)$ is not $\boxtimes$-ID by Theorem \ref{BV}. 
If $0 \leq r <1$ and $s>1$ then $\Im(\log \Sigma_{\bm\pi(r,s)}(x+ \ii 0)) = \pi (1-r)  >0$ for $x>s/(s-1)$, and hence $\bm\pi(r,s)$ is not $\boxtimes$-ID by Theorem \ref{BV}.  
\end{proof}

We use the moment method to prove the weak convergence. Haagerup and M\"oller \cite[Lemma 10]{HM} found a connection between the Mellin transform and the $S$-transform, which is useful for our problem. 

\begin{lem}\label{lemma HM}
Let $\mu$ be a probability measure on $(0,\infty)$. Then
\begin{equation*}
\int_{(0,\infty)} x^\gamma\,\mu(dx) = \frac{1}{B(1-\gamma,1+\gamma)} \int_{(0,1)} \left(\frac{1-x}{x}S_\mu(x-1)\right)^{-\gamma}dx
\end{equation*}
for $\gamma\in(-1,1)$ as an equality in $[0,\infty]$, where $B(p,q)$ is the Beta function. Note that 
$$
\frac{1}{B(1-\gamma,1+\gamma)}=\frac{\sin\pi\gamma}{\pi\gamma}.
$$ 
\end{lem}

\begin{thm}\label{free Bessel} Suppose that either (a) $r\geq0$ and $s=1$, or (b) $r \geq1$ and $s >1$. Then
$$
\D_{t^r}\!\left((\bm\pi(r,s)^{\boxtimes t})^{1/t}\right) \wto \DH_r \quad \text{as}\quad t\downarrow 0. 
$$ 
\end{thm}
\begin{proof} The proof is based on the moment method. The latter case $r \geq1$, $s >1$ is considered firstly. For $\gamma>0$, $t>0$ and $0<\xi<1/\gamma$, we get 
\begin{equation}\label{eq moments}
\begin{split}
&\int_{(0,\infty)} x^\gamma\,\D_{t^{r}}\!\left((\bm\pi(r,s)^{\boxtimes t})^{\xi}\right)(dx)\\
&\quad=t^{r \gamma}\int_{(0,\infty)} x^{\gamma \xi}\,\bm\pi(r,s)^{\boxtimes t}(dx)\\
&\quad=\frac{t^{r \gamma}}{B(1-\gamma \xi,1+\gamma \xi)} \int_{(0,1)} \left(\frac{1-x}{x} S_{\bm\pi(r,s)}(x-1)^{t}\right)^{-\gamma \xi}dx \\
&\quad=\frac{t^{r \gamma}}{B(1-\gamma \xi,1+\gamma \xi)}\int_0^1x^{\gamma \xi (1+ t(r-1) )} (1-x)^{-\gamma \xi}(x+s-1)^{\gamma \xi t} dx \\
&\quad= t^{r \gamma}(s-1)^{\gamma \xi t}\frac{B(\gamma\xi(1+t (r-1) )+1, 1-\gamma \xi)}{B(1-\gamma \xi,1+\gamma \xi)} \times \\
&\qquad \qquad {}_2 F_1(-\gamma \xi t, \gamma\xi  +\gamma \xi t (r-1)+1 ; \gamma \xi t (r-1) +2;  -(s-1)^{-1}) \\
&\quad= t^{r \gamma}(s-1)^{\gamma \xi t}\frac{\Gamma(\gamma\xi + \gamma \xi t (r-1)+1 )}{\Gamma(2+\gamma\xi t (r-1))\Gamma(1+\gamma \xi)} \times \\
&\qquad \qquad {}_2 F_1(-\gamma \xi t, \gamma\xi+\gamma\xi t (r-1) +1; \gamma \xi t (r-1) +2;  -(s-1)^{-1}), 
\end{split}
\end{equation}
where Lemma \ref{lemma HM} was used on the second equality. 
Note that this equality is valid only for $0<\xi < 1/\gamma$ at this moment. 

As a function of $\xi$, the last hypergeometric and gamma functions extend real analytically from $(0,1/\gamma)$ to $(0,\infty)$. On the other hand, the function 
$$
\xi\mapsto \int_{[0,\infty)} x^{\gamma \xi}\,\D_{t^r}\!\left(\bm\pi(r,s)^{\boxtimes t}\right)(dx)
$$
is real analytic in $(0,\infty)$. By analytic continuation, the last expression  \eqref{eq moments} is valid for all $\xi\in(0,\infty)$. 

Now we may put $\xi = 1/t$ and obtain 
\begin{equation}\label{eq moments2}
\begin{split}
&\int_{(0,\infty)} x^\gamma\,\D_{t^{r}}\!\left((\bm\pi(r,s)^{\boxtimes t})^{1/t}\right)(dx)\\
&\quad= \frac{\xi^{-r \gamma}(s-1)^{\gamma}}{\Gamma(2+\gamma(r-1))}\frac{\Gamma(\gamma\xi + \gamma(r-1)+1 )}{\Gamma(\gamma \xi+1)} \times \\
&\qquad \qquad {}_2 F_1(-\gamma, \gamma\xi+\gamma(r-1) +1; \gamma(r-1) +2;  -(s-1)^{-1}). 
\end{split}
\end{equation}
Suppose moreover that $s>2$. By \cite[6.1.47]{AS70} and \cite[15.7.2]{AS70}, we respectively obtain the asymptotic form
\begin{align}
&\frac{\Gamma(\gamma\xi + \gamma(r-1)+1 )}{\Gamma(\gamma \xi+1)} \sim (\gamma \xi)^{\gamma(r-1)}, & \xi\to\infty, \\
&{}_2 F_1(-\gamma, \gamma\xi+\gamma(r-1) +1; \gamma(r-1) +2;  -(s-1)^{-1}) \notag\\
&\quad\quad\quad\sim\frac{\Gamma(\gamma(r-1)+2)}{\Gamma(\gamma r+2)}\left(\frac{\gamma \xi}{s-1}\right)^{\gamma},& \xi\to\infty. \label{asymptoticshyper}
\end{align}
The case $s\in(1,2]$ can also be covered if we use the formula \cite[15.3.4]{AS70} and the asymptotic behavior (\ref{asymptoticshyper}) turn out to be the same. Eventually we obtain for every $\gamma >0$, 
\begin{equation}\label{eq moments3}
\begin{split}
&\int_{(0,\infty)} x^\gamma\,\D_{t^{r}}\!\left((\bm\pi(r,s)^{\boxtimes t})^{1/t}\right)(dx) \to \frac{\gamma^{r \gamma}}{\Gamma(r \gamma +2)} = \int_{[0,\infty)} x^\gamma \DH_r(dx)
\end{split}
\end{equation}
as $t \downarrow 0$. This implies the convergence of moments, and since the limit measure is compactly supported, this implies the weak convergence. 

If $s=1$, then the proof is easier since \eqref{eq moments} is reduced to
\begin{equation*}
\begin{split}\int_{(0,\infty)} x^\gamma\,\D_{t^{r}}\!\left((\bm\pi(r,1)^{\boxtimes t})^{\xi}\right)(dx)
&=\frac{t^{r \gamma}}{B(1-\gamma \xi,1+\gamma \xi)}\int_0^1x^{\gamma \xi (1+ t r )} (1-x)^{-\gamma \xi} dx \\
&= t^{r \gamma}\frac{B(\gamma\xi(1+t r)+1, 1-\gamma \xi)}{B(1-\gamma \xi,1+\gamma \xi)} \\
&= t^{r \gamma} \frac{\Gamma(\gamma\xi(1+t r)+1)}{\Gamma(\gamma \xi t r +2)\Gamma(1+\gamma\xi)}.  
\end{split}
\end{equation*}
This formula is valid for all $\xi\in(0,\infty)$ by analytic continuation. Then one may put $\xi=1/t$ and use the asymptotic form of gamma functions \cite[6.1.47]{AS70} to obtain the convergence of moments (when $\gamma \in \N$). 
\end{proof}

\subsection{Log free stable distributions with index $>1$}  
We find more log free stable distributions in the limit theorem. 
Suppose that $\alpha\in (1,2].$ As an initial probability measure, we take $\bnu_\alpha$ defined by 
\[
S_{\bnu_\alpha}(z)= e^{(-z)^{\alpha-1}}, 
\]
which is $\boxtimes$-ID on $(0,\infty)$ by Theorem \ref{BV}. The measure $\bnu_2$ is compactly supported and has the moment sequence $(\frac{n^n}{n!})_{n=0}^\infty$ (with convention $0^0=1$ as before). This measure already appeared in M\l otkowski \cite{M10} and in a certain limit theorem proved by Sakuma and Yoshida \cite{SY13}.

We need the Laplace transform of the free stable distribution, which was computed in \cite[Theorem 3]{HK14}.  
\begin{lem} \label{laplace}
For $\alpha \in (1,2]$, we have 
$$ 
\int_{-\alpha(\alpha-1)^{1/\alpha-1}}^\infty e^{-\gamma x}\, \bff_{\alpha}(dx)=\sum_{n=0}^\infty \frac{\gamma^{n\alpha}}{\Gamma(2+(\alpha-1)n)n!},\quad \gamma\geq0. 
$$ 
\end{lem}
We are ready to prove the following result. 
\begin{thm} \label{log FS}
For $\alpha\in(1,2]$, the convergence 
\[
(\bnu_\alpha^{\boxtimes t})^{t^{-1/\alpha}} \wto \Law(e^{\bfF_\alpha})  \quad \text{as} \quad t \downarrow 0 
\]
holds, where $\bfF_\alpha$ is a random variable following the law $\bff_\alpha$. The limiting distribution is called a log free $\alpha$-stable law, and in particular, log semicircle distribution if $\alpha=2$. 
\end{thm}

\begin{proof} The proof is based on the moment method.  For $\gamma> 0$ and $\xi\in(0,1/\gamma)$, using Lemma \ref{lemma HM} we get  
\begin{equation}\label{eq05}
\begin{split}
\int_{(0,\infty)} x^{-\gamma}\,(\bnu_\alpha^{\boxtimes t})^{\xi}(dx)
&=\int_{(0,\infty)} x^{-\gamma \xi}\,\bnu_\alpha^{\boxtimes t}(dx)\\
&=\frac{1}{B(1-\gamma \xi,1+\gamma \xi)} \int_{(0,1)} \left(\frac{1-x}{x}S_{\bnu_\alpha}(x-1)^{t}\right)^{\gamma \xi}dx \\
&=\frac{1}{B(1-\gamma \xi,1+\gamma \xi)} \int_{(0,1)} x^{-\gamma \xi}(1-x)^{\gamma \xi}e^{\gamma\xi t(1-x)^{\alpha-1}}dx \\
&= \frac{1}{B(1-\gamma \xi,1+\gamma \xi)} \sum_{n=0}^\infty \frac{(\gamma\xi t)^n}{n!}\int_{(0,1)}x^{-\gamma \xi}(1-x)^{n(\alpha-1)+\gamma \xi}\,dx\\
&= \sum_{n=0}^\infty \frac{(\gamma\xi t)^n}{n!} \frac{\Gamma(1+n(\alpha-1)+\gamma\xi)}{\Gamma(1+\gamma\xi) \Gamma(2+n(\alpha-1))}. 
\end{split}
\end{equation}
By analytic continuation, this formula is valid for any $\xi\in(0,\infty)$. 

Now fix $\gamma \in \N$. The Stirling approximation \cite[6.1.37]{AS70} shows that, for each {\it fixed} $n\in\N$, 
\begin{equation}\label{AS4}
\frac{(\gamma\xi t)^n\Gamma(1+n(\alpha-1)+\gamma\xi)}{\Gamma(1+\gamma\xi)} \sim ((\gamma\xi)^{\alpha}t)^n,\quad \text{as}\quad \xi \to \infty, 
\end{equation} 
and hence putting $\xi=t^{-1/\alpha}$ and letting $t\downarrow0$ imply the convergence
\begin{equation}\label{convergence2}
\int_{(0,\infty)} x^{-\gamma}\,(\bnu_\alpha^{\boxtimes t})^{t^{-1/\alpha}}(dx) \to \sum_{n=0}^\infty \frac{\gamma^{\alpha n}}{n!\Gamma(2+n(\alpha-1))} =\bE[e^{-\gamma \bfF_\alpha}]
\end{equation} provided $\lim_{t\downarrow0}$ and $\sum_{n=0}^\infty$ are exchangeable in \eqref{eq05}. To justify the exchange of limits, we find a dominating function to use Lebesgue's dominated convergence.  
With new variables $x=\xi \gamma=t^{-1/\alpha}\gamma$ and $y=(\alpha-1)n$ such that $x,y \geq1$,   Stirling's formula yields 
\begin{equation}\label{AS3}
\begin{split}
\frac{(\gamma\xi t)^n}{n!} \frac{\Gamma(1+n(\alpha-1)+\gamma\xi)}{\Gamma(1+\gamma\xi) \Gamma(2+n(\alpha-1))}  
&=  \frac{\gamma^{\alpha n}x^{-y} \Gamma(1+x +y)}{n!(y+1) \Gamma(1+x) \Gamma(1+y)} \\
&\leq  C \cdot \frac{\gamma^{\alpha n} x^{-y}  e^{-x-y} (x+y)^{x+y+1/2}}{n! e^{-x} x^{x+1/2} e^{-y} y^{y+1/2}} \\
&= C \cdot \frac{\gamma^{\alpha n} (x+y)^{x+y+1/2}}{n! x^{x+y+1/2} y^{y+1/2}} \\
&= C \cdot \frac{\gamma^{\alpha n} }{n! } \cdot \left(1+\frac{y}{x}\right)^x \cdot \left(\frac{1}{y}+\frac{1}{x}\right)^{y+1/2}
\end{split}
\end{equation}
for some constant $C>0$. We use here the inequalities 
$$
\left(1+\frac{y}{x}\right)^x \leq e^y =e^{(\alpha-1) n}
$$
and 
$$
\left(\frac{1}{y}+\frac{1}{x}\right)^{y+1/2} \leq 2^{y+1/2} = 2^{(\alpha-1)n +1/2}. 
$$
Therefore, the LHS of \eqref{AS3} is bounded by the summable sequence 
\begin{equation}\label{AS5}
\sqrt{2}C \cdot \frac{(2^{\alpha-1} e^{\alpha-1}\gamma^\alpha)^n }{ n! }, 
\end{equation}
which does not depend on $t$. By Lebesgue's convergence theorem, the convergence \eqref{convergence2} holds for every $\gamma \in\N$.  

 This shows that all moments of $(\bnu_\alpha^{\boxtimes t})^{-t^{-1/\alpha}}$ converge to those of $\Law(e^{-\bfF_\alpha})$. Since $\bff_\alpha$ has a support bounded below, the measure $\Law(e^{-\bfF_\alpha})$ is compactly supported and hence $(\bnu_\alpha^{\boxtimes t})^{-t^{-1/\alpha}} \wto \Law(e^{-\bfF_\alpha})$. Finally taking the inverse we obtain the result. 
\end{proof}

The above method can be generalized to a larger class of initial distributions $\mu$. 

\begin{thm} \label{multi log FS}
Suppose that $k\in\N$, $2 \geq \alpha_1 > \cdots > \alpha_k >1$ and $p_1, \dots, p_k>0 $, and define 
$$
\bnu(\bm \alpha, \mathbf{p}):= \bnu_{\alpha_1}^{\boxtimes p_1}  \boxtimes \cdots \boxtimes \bnu_{\alpha_k}^{\boxtimes p_k},  
$$
where $\bm\alpha=(\alpha_1,\dots, \alpha_k)$ and $\mathbf{p}=(p_1,\dots, p_k)$. 
Then 
\[
(\bnu(\bm \alpha, \mathbf{p})^{\boxtimes t})^{t^{-1/\alpha_1}} \wto \Law(e^{p_1^{1/\alpha_1}\bfF_{\alpha_1} }) \quad \text{as}\quad t \downarrow 0. 
\]
\end{thm}
\begin{proof} The proof is similar, so only the required changes are noted below. Let $\N_0:= \N \cup\{0\}$. Then \eqref{eq05} changes to 
\begin{equation}\label{Mellin}
\begin{split}
&\int_{(0,\infty)} x^{-\gamma}\,(\bnu(\bm \alpha, \mathbf{p})^{\boxtimes t})^{\xi}(dx) \\
&\qquad = \sum_{(n_1,\dots, n_k)\in\N_0^k} \left(\prod_{i=1}^k\frac{(\gamma\xi t p_i)^{n_i}}{n_i!}\right) \frac{\Gamma(1+\sum_{i=1}^k n_i(\alpha_i-1)+\gamma\xi)}{\Gamma(1+\gamma\xi) \Gamma(2+\sum_{i=1}^k n_i(\alpha_i-1))},  
\end{split}
\end{equation}
which holds for $\xi \in(0,\infty)$ by analytic continuation. If we put $\xi= t^{-1/\alpha_1}$ then the Stirling formula shows that, for each fixed $(n_1,\dots, n_k)$,   
\begin{equation}\label{}
\begin{split}
&\left(\prod_{i=1}^k (\gamma\xi t p_i)^{n_i}\right) \frac{\Gamma(1+\sum_{i=1}^k n_i(\alpha_i-1)+\gamma\xi)}{\Gamma(1+\gamma\xi)} \\
&\sim \left(\prod_{i=1}^k (\gamma\xi t p_i)^{n_i}\right) (\gamma \xi)^{\sum_{i=1}^kn_i(\alpha_i-1)}  \\ 
&= \left(\prod_{i=1}^k (\gamma p_i^{1/\alpha_i})^{\alpha_i n_i}\right) \prod_{i=2}^k t^{\frac{n_i(\alpha_1-\alpha_i)}{\alpha_1}} \quad \text{as}\quad t\downarrow0,  
\end{split}
\end{equation} 
which tends to 0 if $(n_2, \dots, n_k) \neq (0,\dots ,0)$. This shows that \eqref{Mellin} converges to $\bE[e^{-\gamma p_1^{1/\alpha_1}\bfF_{\alpha_1}}]$, supposing that the limits $\lim_{t\downarrow0}$ and $\sum_{(n_1,\dots, n_k)\in \N_0^k}$ are exchangeable. The exchange of the limits can be verified by an estimate similar to \eqref{AS5} with $x= \gamma t^{-1/\alpha_1}$ and $y = \sum_{i=1}^k n_i(\alpha_i-1)$. 
\end{proof}

\subsection{Further examples}
We find more examples of convergence to log free stable distributions by taking the multiplicative free convolution with Boolean stable distributions. We exploit several identities obtained in \cite{AH16}. 
We start from an obvious property which shows that the dilation and power of limit distributions are also limit distributions. 
\begin{prop}\label{power} Let $\mu$ be a $\boxtimes$-ID measure on $(0,\infty)$ and $\nu$ be a probability measure on $(0,\infty)$. Let $a, b \colon (0,\infty)\to (0,\infty)$ be functions such that  
$\D_{b(t)}\!\left((\mu^{\boxtimes t})^{a(t)}\right)\wto\nu$ as $t\downarrow0$. Then for any $r \in \R$ and $s>0$, 
$$
\D_{s b(t)^r}\!\left((\mu^{\boxtimes t})^{a(t)r}\right) \wto \D_s(\nu^r) \quad \text{as}\quad t\downarrow0. 
$$
\end{prop}

Now we find more nontrivial examples of limit theorems using Boolean stable laws and some identities obtained in \cite{AH16}.

\begin{thm}\label{booleanmixture}
 Assume that $\mu,\nu$ are probability measures on $(0,\infty)$ and $\mu$ is $\boxtimes$-ID. Let $\alpha \in(0,1)$, $\beta \in [-1, \infty)$, $\gamma \in \R$ and  $a, b \colon (0,\infty)\to (0,\infty)$ be measurable functions such that 
\begin{enumerate}[\quad\rm(1)]
\item $\D_{b(t)}\!\left((\mu^{\boxtimes t})^{a(t)}\right)\wto\nu$ as $t\downarrow0$,
\item\label{AS1} $a(t) =t^\beta \!\left(1+o(|\log t|^{-1})\right)$ as $t\downarrow0$, 
\item the function $b(t)$ is regularly varying of index $\gamma$ as $t\downarrow0$. 
\end{enumerate}
 Then as $t\downarrow0$, 
$$
\D_{b(t)}\!\left((\bfb_{\alpha}\boxtimes \mu)^{\boxtimes t}\right)^{a(t)}\wto 
\begin{cases} 
\Law(e^{\bfC_{0,(1-\alpha)\pi/\alpha}})\circledast\nu, & \text{if}\quad \beta=-1, \\
\nu, & \text{if} \quad  \beta \in(-1,\infty).  
\end{cases} 
$$
\end{thm}
\begin{proof}
Recall a formula in \cite[Proposition 3.7]{AH16}: 
\begin{equation}
(\bfb_{\alpha})^{\boxtimes t} = \bfb_{\ff}.  
\end{equation}
We define the function 
\begin{equation}
f(t):= \frac{(\alpha+t(1-\alpha))a(t)}{\alpha a\!\left(\ff\right)}. 
\end{equation}
Then 
\begin{equation}
\begin{split}
&\D_{b\left(\ff\right)^{f(t)}}\!\left((\bfb_{\alpha}\boxtimes \mu)^{\boxtimes t}\right)^{a(t)} \\
&= \D_{b\left(\ff\right)^{f(t)}}\!\left(\bfb_{\ff}\boxtimes (\mu^{\boxtimes \ff })^{\boxtimes \frac{\alpha + (1-\alpha)t}{\alpha}}\right)^{a(t)} \\
&= \D_{b\left(\ff\right)^{f(t)}}\!\left(\bfb_{\ff}\circledast \left(\mu^{\boxtimes \ff}\right)^{\frac{\alpha+(1-\alpha)t}{\alpha}}\right)^{a(t)} \\
&= \left(\bfb_{\alpha}^{\boxtimes t}\right)^{a(t)}\circledast \left(\D_{b\left(\ff\right)}\left(\mu^{\boxtimes \ff}\right)^{a\left(\ff\right)}\right)^{f(t)},
\end{split}
\end{equation}
where \cite[Theorem 4.5]{AH16} was used on the second equality. We know that $\left(\bfb_{\alpha}^{\boxtimes t}\right)^{1/t} \wto \Law(e^{\bfC_{0,(1-\alpha) \pi/\alpha}})$ from Example \ref{LC BS}, and therefore we obtain by Lemma \ref{lemma limit2}, 
\begin{equation}
\lim_{t\downarrow0}\left(\bfb_{\alpha}^{\boxtimes t}\right)^{a(t)}=
\begin{cases}
\Law(e^{\bfC_{0,(1-\alpha) \pi/\alpha}}), &  \text{if}\quad \beta=-1, \\
\delta_1,&  \text{if}\quad \beta \in (-1,\infty).  
\end{cases}
\end{equation}
In view of Lemma \ref{lemma limit2} it suffices to show that 
\begin{equation}\label{convergence}
f(t)\to 1 \quad \text{and} \quad  \frac{b\!\left(\ff\right)^{f(t)}}{b(t)} \to 1 \quad \text{as}\quad  t\downarrow0. 
\end{equation}
Firstly, \eqref{AS1} implies that 
\begin{equation}\label{asymptotics-f}
f(t) = 1+o(|\log t|^{-1})  \quad \text{as} \quad t\downarrow0.
\end{equation}
For the second convergence in \eqref{convergence}, we use Karamata's representation \cite[Theorem 1.3.1 or (1.5.2)]{BGT87}, 
$$
b(t) = \exp\left(\theta(t)+ \int_t^T\frac{\psi(s)}{s}ds\right), \qquad t \in(0, T],
$$
for some constant $T>0$ and some bounded measurable functions $\theta, \psi\colon(0,T]\to\R$ such that $\lim_{t\downarrow0}\theta(t)=\theta_0 \in \R$ (and $\lim_{t\downarrow0}\psi(t)=\gamma$, which is not needed in the proof). Considering $\alpha t/(\alpha+(1-\alpha)t) = t(1+o(1))$ and \eqref{asymptotics-f}, we get 
\begin{equation}\label{}
\begin{split}
&\log \frac{b\!\left(\ff\right)^{f(t)}}{b(t)}  \\
&= \underbrace{f(t)\theta(t+ o(t)) - \theta(t)}_{=:I_1(t)} + \underbrace{\left(1+o\!\left(\frac{1}{|\log t|}\right)\right)\int_{t+o(t)}^T \frac{\psi(s)}{s}ds -  \int_t^T\frac{\psi(s)}{s}ds.}_{=:I_2(t)} 
\end{split}
\end{equation}
It is easy to see that $I_1(t) \to 0$.  For the second term $I_2$, 
\begin{equation}\label{estimate1}
\begin{split}
|I_2(t)| 
&= \left|\int_{t+o(t)}^t \frac{\psi(s)}{s}ds +  o\!\left(\frac{1}{|\log t|}\right) \int_{t+o(t)}^T \frac{\psi(s)}{s} ds \right| \\
&\leq \frac{\|\psi\|_\infty}{t+o(t)} o(t)  +   o\!\left(\frac{1}{|\log t|}\right) \|\psi\|_\infty \log \frac{T}{t+o(t)}  \to 0.  
\end{split}
\end{equation}
This establishes \eqref{convergence}. 
\end{proof}

We can then find more examples of probability measures which yield log free stable distributions.  
\begin{cor}\label{cor LFS}  Let $\beta \in(0,1)$. Following the notations in Theorem \ref{multi log FS}, we have   
\[
((\bfb_\beta \boxtimes \bnu(\bm \alpha, \mathbf{p}))^{\boxtimes t})^{t^{-1/\alpha_1}} \wto \Law(e^{p_1^{1/\alpha_1}\bfF_{\alpha_1}}) \quad \text{as} \quad t\downarrow0. 
\]
\end{cor}
We deduce another corollary of Theorem \ref{booleanmixture}. In \cite{HM} Haagerup and M\"{o}ller considered the probability measures $\mu_{\alpha,\beta}$ defined by 
\begin{equation}
S_{\mu_{\alpha,\beta}}(z)=\frac{(-z)^{\alpha}}{(1+z)^\beta},\quad \alpha,\beta \geq0. 
\end{equation}
Its probability density function has an implicit expression. Computing the $S$-transform shows that  
\begin{equation}
\mu_{\alpha,\beta}=
\begin{cases} 
\bfb_{\frac{1}{1+\alpha}}\boxtimes \bm\pi^{\boxtimes(\beta-\alpha)},&\alpha \leq \beta,\\
\bfb_{\frac{1}{1+\beta}}\boxtimes \bff_{\frac{1}{1+\alpha-\beta}},&\alpha\geq\beta, 
\end{cases}
\end{equation}
and in particular $\mu_{\alpha,\beta}$ is $\boxtimes$-ID for any $\alpha,\beta\geq0$. We may restrict to the case $\alpha \leq \beta$ since for $\alpha > \beta$ the identity 
\begin{equation}
(\mu_{\alpha, \beta})^{-1}  = \mu_{\beta,\alpha} 
\end{equation}
holds, 
which can be verified by $S$-transform and the formula 
\begin{equation}
S_{\mu^{-1}}(z) = \frac{1}{S_\mu(-1-z)}
\end{equation}
 for a probability measure $\mu$ on $(0,\infty)$ (see \cite[Proposition 3.13]{HS07}). 
 Recall from Theorem \ref{free Bessel} that $\D_{t^{\beta-\alpha}}\left(\left(\mu^{\boxtimes t}\right)^{1/t}\right)\wto \DH_{\beta-\alpha},~t\downarrow0$ for $\mu=\bm\pi^{\boxtimes (\beta-\alpha)}$. Now Theorem \ref{booleanmixture} implies the following result. 
\begin{cor}\label{cor DH} For $0 \leq \alpha \leq \beta$, we have the convergence 
$$
\D_{t^{\beta-\alpha}}\left(\left(\mu_{\alpha,\beta}^{\boxtimes t}\right)^{1/t}\right)\wto \Law(e^{\bfC_{0,\alpha\pi}})\circledast\DH_{\beta-\alpha} \quad \text{as}\quad t\downarrow0. 
$$
\end{cor}
Using Proposition \ref{prop DH} shows that 
\begin{equation}
\Law(e^{\bfC_{\beta, \gamma}})\circledast (\DH_1)^{a} = \Law(e^{\bfC_{\beta, \gamma} + a \bfF_1 }), \qquad \beta \in\R, \gamma \geq0, a \in \R, 
\end{equation}
where the random variables $\bfC_{\beta, \gamma}$ and $ \bfF_1$ are assumed to be independent. Moreover, assuming free independence of $\bfC_{\beta, \gamma}$ and $ \bfF_1$ gives the same distribution, thanks to the speciality of the Cauchy law  \eqref{special Cauchy}. Since $\Law(\bfC_{\beta, \gamma} + a \bfF_1)$ covers all free $1$-stable laws, considering Theorem \ref{log FS}, Corollary \ref{cor DH} and Proposition \ref{power}, we have obtained a certain class of possible limit distributions.  
\begin{thm} \label{limit}
Any probability measure in the family 
$$
\{\Law(e^{u \bfF_\alpha +v}) \mid \alpha \in (1,2], u,v \in \R\} \cup \{\text{\rm Log free $1$-stable distributions}\}
$$ 
appears in the limit theorem of the form \eqref{small time}. 
\end{thm}
Note that the above probability measures are all log free stable with index $\geq1$. 

\begin{prob}
Determine all the possible limit distributions of \eqref{small time}. In particular, determine whether the following distributions can appear in the limit theorem: 
\begin{itemize}
\item log free stable laws with index $>1$ and with an arbitrary asymmetry parameter $\rho$;  
\item log free stable laws with index $<1$;  
\item probability measures which are not log free stable laws. 
\end{itemize} 
\end{prob}

\begin{prob}[Domain of attraction]
Characterize initial probability measures $\mu$ such that \eqref{small time} converges to a given non-degenerate distribution (e.g.\ probability measures in Theorem \ref{limit}) for some functions $a,b\colon (0,\infty) \to (0,\infty)$. Does a transfer principle (like in Theorem \ref{thm ALP}) holds between free and classical limit theorems? 
\end{prob}

\section{Positive multiplicative Boolean L\'evy processes at small time} \label{S5}
As mentioned in Section \ref{sec MBC}, the Boolean power $\mu^{\utimes t}$ is well defined for $0 \leq t \leq1$ and for any probability measure $\mu$ on $[0,\infty)$. Therefore, one may discuss the convergence of 
\begin{equation}\label{small time B}
\D_{b(t)}(\mu^{\utimes t})^{a(t)}, \qquad t\downarrow0, 
\end{equation}
where $a,b\colon (0,1] \to (0,\infty)$ are functions. We can give a more solid solution to this problem than the free case since the analysis is easier.  

The defining relation \eqref{eq MBCP} for the Boolean convolution power, combined with 
\begin{equation}
\eta_\mu(z) = 1-z F_\mu(1/z) = 1-\frac{z}{G_\mu(1/z)}, 
\end{equation}
yields that 
\begin{equation}\label{CauchyTransform}
G_{ \mu^{\sutimes t} }(z) =  \frac{1}{z - (z\eta_\mu(1/z))^t}=\frac{1}{z - (z- F_\mu(z))^t}. 
\end{equation}
We consider the following assumption on $\mu$: 
\begin{itemize} 
\vspace{3mm}
\item[\AS]  There exists an open interval $I \subset (0,\infty)$ such that $1 \in I$ and the limit 
$$
F_\mu(x):= F_\mu(x+\ii0) := \lim_{y\downarrow0}F_\mu(x +\ii y) \in \C^+ \cup \R
$$
exists for each $x\in I$, and the map $F_\mu\colon I \to  \C^+ \cup \R$ is continuous at $1$. 
\vspace{3mm}
\end{itemize}
A sufficient condition for \AS is the existence of a H\"older continuous density around $x=1$; see Example \ref{exa Holder}. 
The assumption \AS, equation \eqref{CauchyTransform} and Stieltjes inversion imply that $\mu^{\utimes t}$ is Lebesgue absolutely continuous on $I$ and 
\begin{equation}
\frac{d\mu^{\utimes t}}{dx} = \frac{1}{\pi} \Im\!\left(\frac{1}{(x- F_\mu(x+\ii0))^t-x} \right), \qquad x \in I, 0<t<1, 
\end{equation}
unless the denominator is zero. Moreover, for $0<t<1$ and $s>0$, the probability measure $(\mu^{\utimes t})^{1/s}$ is Lebesgue absolutely continuous on $I^{1/s} := \{x \in (0,\infty): x^s \in I\}$ with density
\begin{equation}\label{density B}
\frac{d(\mu^{\utimes t})^{1/s}}{dx}=   \frac{s}{\pi x} \Im\!\left(\frac{1}{x^{-s} (x^s - F_\mu(x^s+\ii0))^t -1} \right), \qquad x \in I^{1/s}
\end{equation}
unless the denominator is zero.

\subsection{Log Cauchy distribution}\label{sec B}
We first consider the log Cauchy distributions. 
\begin{thm}\label{LCB} Let $\mu$ be a probability measure on $[0,\infty)$ satisfying \AS and $F_\mu(1) \in \C^+ \cup (1,\infty)$, and so we may write $\log(1- F_\mu(1) -\ii0) = \beta - \ii \gamma$, where $(\beta,\gamma) \in \R \times (0, \pi]$. Then the convergence  
\[
\frac{d(\mu^{\utimes t})^{1/t}}{d x}\to \frac{\gamma}{\pi x [(\log x - \beta)^2 + \gamma^2]} \quad \text{as}\quad t \downarrow0
\] 
holds uniformly on each compact set of $(0,\infty)$. In particular, $(\mu^{\utimes t})^{1/t}$ converges to $\Law(e^{\mathbf{C}_{\beta,\gamma}})$ weakly. 
\end{thm}
\begin{rem} It is notable that the parameter $\gamma$ is less than or equal to $\pi$, while it was an arbitrary positive number in the free case in Theorem \ref{LCF}.  
\end{rem}
\begin{proof} Take any compact set $K$ of $(0,\infty)$. Then $x^t\in I$ for sufficiently small $t\in(0,1)$ and any $x\in K$, and hence the density formula \eqref{density B} is valid on $K$ unless the denominator is zero. Note that $x^t = 1 + t \log x + o(t)$ as $t\downarrow0$ by calculus and $F_\mu(x^t+\ii0) = w +  o(1)$ as $t\downarrow0$ uniformly on $x \in K$ by \AS. Then 
\begin{equation}\label{eq LCB}
\begin{split}
\frac{d(\mu^{\utimes t})^{1/t}}{d x}
&= \frac{t}{\pi x} \Im\!\left(\frac{1}{(1-t \log x+o(t)) (1 - w -\ii0 +o(1))^t -1} \right)   \\
&= \frac{t}{\pi x} \Im\!\left(\frac{1}{(1-t \log x +o(t)) (1 - w-\ii0)^t(1 +o(1))^t -1} \right)   \\
&= \frac{t}{\pi x} \Im\!\left(\frac{1}{(1-t \log x +o(t)) (1 + t \log(1 - w-\ii0) +o(t))(1 +o(t)) -1} \right)   \\
&= \frac{1}{\pi x} \Im\!\left(\frac{1}{\log(1 - w-\ii0) -\log x +o(1)} \right)   \\
&\to \frac{1}{\pi x} \Im\!\left(\frac{-1}{\log x - \beta + \ii\gamma} \right) \quad  \text{as} \quad t \downarrow0. 
\end{split}
\end{equation}
This convergence is uniform on $K$. The weak convergence follows from Lemma \ref{convergence1} with $B=(0,\infty)$. 
\end{proof}

\begin{exa} \label{exa Holder}
Suppose that $\mu$ is Lebesgue absolutely continuous in a finite open interval $I$ containing the point $1$, and $d\mu/dx$ is strictly positive and locally $\rho$-H\"{o}lder continuous on $I$ for some $0< \rho<1$. Then the assumption \AS  is satisfied and $F_\mu(1) \in \C^+$. Therefore, $\gamma \in (0,\pi)$ and the convergence of Theorem \ref{LCB} holds. 

The proof is as follows. In the decomposition 
$$
G_\mu(z) = \int_{I} \frac{1}{z- u} \mu(d u) + \int_{I^c} \frac{1}{z- u} \mu(d u) =: G_1(z)+G_2(z), 
$$
the second part $G_2$ extends continuously to $\C^+ \cup I$, taking real-values on $I$. Considering 
$$
G_1(x+\ii y) = \int_I \frac{x-u}{(x-u)^2 + y^2} \mu(d u) - \ii   \int_I \frac{y}{(x-u)^2 + y^2} \mu(d u)
$$
and \cite[Lemmas $\alpha,\beta,\delta$]{Tit26} (with some modification of proofs because we only assume the local H\"older continuity, not the global one), the Cauchy transform $G_1$ extends to a continuous function on $\C^+ \cup I$ and 
$$
G_1(x) = p.v.\int_I \frac{1}{x-u} \mu(d u) - \ii\pi \frac{d\mu}{dx}
$$
on $I$. 
The real part is locally $\rho$-H\"older continuous \cite[3.36]{Tit26} in $I$, and so $G_1(x)$ is continuous in $I$. Since $\Im(G_1(1)) = -\pi \frac{d\mu}{dx}\big|_{x=1}<0$, it follows that $F_\mu(1)\in\C^+$. 
\end{exa}

\begin{exa} Let $\mu= \frac{1}{2}(\delta_2 + \delta_p)$ for $p \in (0,\infty)$. Then 
$$
F_\mu(z) = \frac{(z-2)(z-p)}{z- 1- p/2}. 
$$
Hence $F_\mu$ satisfies \AS.  The condition $F_\mu(1) = 2(1-p)/p>1$ is satisfied if and only if $0<p<2/3$. If this condition is satisfied then $\beta -\ii\gamma=\log \frac{2-3 p}{p} - \ii \pi$, and hence
$$
(\mu^{\utimes t})^{1/t} \wto \frac{1}{x [(\log \frac{px}{2-3 p} )^2 + \pi^2]} \quad \text{as}\quad t \downarrow0. 
$$ 
\end{exa}

\subsection{Log Boolean stable distributions with index $<1$}

The distribution $\Law(e^{\mathbf{B}_{\alpha,\rho,r}})$ is called the log Boolean stable law, where $\mathbf{B}_{\alpha,\rho,r}$ is a random variable following the law $\bfb_{\alpha, \rho, r}$. 
The convergence to log Boolean stable distributions is shown below.  
\begin{thm}\label{LBS} Let $\mu$ be a probability measure on $[0,\infty)$ satisfying \AS, and for some $\alpha \in(0,1), \rho \in [0,1]$ and $r>0$,  
\begin{equation}\label{AS2}
F_\mu(x) = r e^{\ii \alpha\rho\pi} (x-1+\ii0)^{1-\alpha} + o(|x-1|^{1-\alpha})  \quad \text{as}\quad x \to 1. 
\end{equation}
Then the convergence  
\[
\frac{d(\mu^{\utimes t})^{t^{-1/\alpha}}}{d x} \to \frac{ r \sin \alpha \rho\pi }{\pi x} \cdot \frac{ (\log x)^{\alpha-1} }{(\log x)^{2\alpha} + 2 r (\cos \alpha\rho\pi) (\log x)^\alpha + r^2} \quad \text{as}\quad t \downarrow0
\] 
holds uniformly on each compact set of $(1,\infty)$, and 
\[
\frac{d(\mu^{\utimes t})^{t^{-1/\alpha}}}{d x} \to \frac{r  \sin \alpha(1- \rho)\pi }{\pi x} \cdot \frac{ (-\log x)^{\alpha-1} }{(-\log x)^{2\alpha} + 2 r (\cos \alpha(1-\rho)\pi) (-\log x)^\alpha + r^2} \quad \text{as}\quad t \downarrow0
\] 
holds uniformly on each compact set of $(0,1)$. 
 In particular, $(\mu^{\utimes t})^{t^{-1/\alpha}}$ converges to $\Law(e^{\mathbf{B}_{\alpha,\rho, r}})$ weakly. 
\end{thm}
\begin{rem} \begin{enumerate}[\rm(1)]
\item The asymptotic behavior \eqref{AS2} is equivalent to $G_\mu(x) = (1/r) e^{-\ii \alpha\rho\pi} (x-1+\ii0)^{\alpha-1} + o(|x-1|^{\alpha-1})$. By Stieltjes inversion, we obtain that 
\begin{equation}\label{AS6}
\frac{d\mu}{dx} = 
\begin{cases} 
\displaystyle \frac{\sin\alpha \rho\pi}{\pi r} (x-1)^{\alpha-1}+o((x-1)^{\alpha-1}), & x\downarrow1, \\[4mm]
\displaystyle \frac{\sin\alpha (1-\rho)\pi}{\pi r}(1-x)^{\alpha-1}+o((1-x)^{\alpha-1}), & x \uparrow1. 
\end{cases}
\end{equation}
Hence the triplet $(\alpha, \rho, r)$ can be determined from a local behavior of $d\mu/d x$ at $1$. Conversely, it is not known if the asymptotic behavior \eqref{AS6} of $d\mu/dx$ implies the asymptotic behavior \eqref{AS2} of $F_\mu$ (or $G_\mu$). When $d\mu/dx$ satisfies an analytic property then the converse  is true, see Example \ref{Analytic}.  
\item While the Cauchy distribution is a Boolean $1$-stable law, we cannot unify Theorems \ref{LCB} and \ref{LBS}. This is because the estimate \eqref{estimate2} below fails to hold for $\alpha=1$. 
\end{enumerate}
\end{rem}
\begin{proof} We define $\theta := \alpha \rho \pi$. Take any compact set $K_1$ of $(1,\infty)$. Then $x^t\in I$ for sufficiently small $t<1$ and any $x\in K_1$, and hence the density formula \eqref{density B} is valid on $K_1$ unless the denominator is zero. Note that 
\begin{equation}\label{}
\begin{split}
x^{t^{1/\alpha}}- F_\mu(x^{t^{1/\alpha}}+\ii0) 
&= 1 + t^{1/\alpha}\log x +o(t^{1/\alpha}) -r e^{\ii\theta}(t^{1/\alpha} \log x + o(t^{1/\alpha})+ \ii0)^{1-\alpha}  \\
&= 1 -r e^{\ii\theta} t^{(1-\alpha)/\alpha} (\log x)^{1-\alpha} + o(t^{(1-\alpha)/\alpha})
\end{split}
\end{equation}
as $t\downarrow0$ uniformly on $x \in K_1$. By calculus we obtain 
\begin{equation}\label{estimate2}
\begin{split}
(x^{t^{1/\alpha}}- F_\mu(x^{t^{1/\alpha}}+\ii0) )^t 
&=1 -r e^{\ii\theta} t^{1/\alpha} (\log x)^{1-\alpha} + o(t^{1/\alpha}). 
\end{split}
\end{equation}
 Therefore, 
\begin{equation}
\begin{split}
\frac{d(\mu^{\utimes t})^{t^{-1/\alpha}}}{d x}
&= \frac{t^{1/\alpha}}{\pi x} \Im\!\left(\frac{1}{[1-t^{1/\alpha} \log x+o(t^{1/\alpha})] [1 -r e^{\ii\theta} t^{1/\alpha} (\log x)^{1-\alpha} + o(t^{1/\alpha}) ]-1} \right)   \\
&= \frac{1}{\pi x} \Im\!\left(\frac{1}{-\log x -r e^{\ii\theta} (\log x)^{1-\alpha} + o(1) } \right)   \\
&\to \frac{1}{\pi x} \Im\!\left(\frac{-1}{\log x +r e^{\ii\theta} (\log x)^{1-\alpha}} \right)  \quad  \text{as} \quad t \downarrow0 \\
&= \frac{r \sin \theta}{\pi x}\cdot \frac{(\log x)^{\alpha-1}}{(\log x)^{2\alpha} +2 r (\cos \theta) (\log x)^{\alpha} +r^2}. 
\end{split}
\end{equation}
The convergence is uniform on $K_1$. 

Take a compact set $K_2 \subset (0,1)$. Note that for $x<1$, 
$$
F_\mu(x) = r e^{\ii (\theta+\pi (1-\alpha))} (1- x)^{1-\alpha} + o((1-x)^{1-\alpha}).
$$
Hence, \eqref{estimate2} holds true if we replace $e^{\ii\theta}$ by $e^{\ii (\theta+\pi (1-\alpha))}$ and $\log x$ by $-\log x$: 
 \begin{equation}\label{estimate3}
\begin{split}
(x^{t^{1/\alpha}}- F_\mu(x^{t^{1/\alpha}}+\ii0) )^t 
&=1 + r e^{\ii(\theta-\alpha\pi)} t^{1/\alpha} (-\log x)^{1-\alpha} + o(t^{1/\alpha})
\end{split}
\end{equation}
uniformly on $K_2$. Hence 
\begin{equation}
\begin{split}
\frac{d(\mu^{\utimes t})^{t^{-1/\alpha}}}{d x}
&= \frac{t^{1/\alpha}}{\pi x} \Im\!\left(\frac{1}{[1-t^{1/\alpha} \log x+o(t^{1/\alpha})] [1 + r e^{\ii(\theta-\alpha\pi)} t^{1/\alpha} (-\log x)^{1-\alpha} + o(t^{1/\alpha})]-1} \right)   \\
&\to \frac{1}{\pi x} \Im\!\left(\frac{1}{-\log x +r e^{\ii(\theta-\alpha\pi)} (-\log x)^{1-\alpha}} \right)  \quad  \text{as} \quad t \downarrow0 \\
&= \frac{r \sin (\alpha\pi-\theta)}{\pi x}\cdot \frac{(-\log x)^{\alpha-1}}{(-\log x)^{2\alpha} +2 r \cos (\alpha\pi-\theta) (-\log x)^{\alpha} +r^2}
\end{split}
\end{equation}
uniformly on $K_2$. The limiting function is the probability density function of $\Law(e^{\mathbf{B}_{\alpha,\rho}})$. The weak convergence follows from Lemma \ref{convergence1} with $B=(0,\infty) \setminus\{1\}$. 
\end{proof}

\begin{exa}\label{Analytic} Suppose that $\alpha \in (0,1), c_1,c_2 \geq0, c_1+c_2>0, \delta>0$ and $\mu$ is a Borel probability measure such that 
$\mu|_{(1-\delta, 1+\delta)}$ has a local density function $p(x)$  
 of the form 
\begin{equation}\label{eq910}
p(x)= 
\begin{cases}
c_1(x-1)^{\alpha-1}(1+f_1(x)), & 1<x< 1 + \delta, \\
c_2(1-x)^{\alpha-1}(1+f_2(x)), & 1 - \delta < x <1, 
\end{cases}
\end{equation}
where $f_k$ is analytic in a neighborhood of $1$ and $f_k(0)=0$, $k=1,2$ (the assumption of analyticity of $f_k$ can be weakened slightly). 
From the proof of \cite[Theorem 5.1, (5.6)]{Has14}, for some $\beta_1\geq 0$
\begin{equation}
\int_{(1,1+\delta)} \frac{1}{z-x} p(x)\,dx= -\beta_1(1-z)^{\alpha-1}+o(|1-z|^{\alpha-1}) \quad \text{as} \quad z \to 1, 
\end{equation}
uniformly on $z \in \C^+$. Considering the symmetry, we obtain  for some $\beta_2 \geq0$, 
\begin{equation}
\int_{(1-\delta,1)} \frac{1}{z-x} p(x)\,dx= \beta_2(z-1)^{\alpha-1}+o(|1-z|^{\alpha-1}) \quad \text{as} \quad z \to 1.  
\end{equation}
Combining these two asymptotic behaviors gives 
\begin{equation}
G_\mu(z)= (\beta_1 e^{-\alpha\pi \ii} + \beta_2)(z-1)^{\alpha-1}+o(|1-z|^{\alpha-1}) \quad \text{as} \quad z \to 1, 
\end{equation}
and hence the assumption \eqref{AS2} of Theorem \ref{LBS} is satisfied. Since $c_1+c_2>0$, the Stieltjes inversion implies that $\beta_1 + \beta_2>0$ too. 

\end{exa}

So far we have obtained limit theorems converging to log Boolean stable laws (including the log Cauchy as index 1), and described their domains of attraction. An unsolved problem is: 

\begin{prob} Are there non-degenerate limit distributions \eqref{small time B} except log Boolean stable laws with index $\leq 1$? 
\end{prob}

\section{Proof of Theorem \ref{LCF}}\label{sec6} 

The convergence in distribution of positive MFLPs to the log Cauchy distribution can be reduced to the easier problem of MBLPs, the latter of which was discussed in Section \ref{sec B}. 
However, we need a framework of free and Boolean convolutions beyond convolutions of probability measures. This framework is developed below, and in particular, we generalize concepts and results introduced in \cite{AH13,BB05}.

\subsection{Convolutions of maps on the negative half-line}
\begin{defi}
Let $\cE$ be the set of maps $\eta\colon(-\infty,0)\to(-\infty,0)$ of the form 
\begin{equation}\label{eta-u}
\eta(x) = x \exp[-u(x)], 
\end{equation}
where $u\colon (-\infty,0)\to \R$ is a continuous non-increasing function. 
\end{defi}

This class generalizes the class of $\eta$-transforms of probability measures on $[0,\infty)$ which is not $\delta_0$.  

\begin{prop} 
If $\mu \neq \delta_0$ is a probability measure on $[0,\infty)$ then $\eta_\mu|_{(-\infty,0)} \in\cE$. 
\end{prop}
\begin{proof} The Pick-Nevanlinna representation \eqref{PN} of $F_\mu$ shows that 
$$
z\mapsto\frac{\eta_\mu(z)}{z} = \frac{1}{z} - F_\mu\left(\frac{1}{z}\right)
$$
is an analytic map from $\C\setminus[0,\infty)$ into $\C$, and maps $\C^-$ into $\C^+\cup(0,\infty)$. Its principal logarithm can therefore be defined as an analytic map from $\C^-$ to $\C^+\cup\R$, and hence has the Pick-Nevanlinna representation 
$$
u(z):=\log\frac{\eta_\mu(z)}{z}  = -a z + b + \int_{[0,\infty)} \frac{1+z t}{z-t}\, \sigma(d t)  
$$  
 for some $a \geq0, b\in\R$ and a nonnegative finite measure $\sigma$ on $[0,\infty)$. By calculus we see that $u'(x)\leq0$ for $x<0$. 
\end{proof}

\begin{defi}
Given $\eta \in \cE$ and $s\geq0$, we define a {\it multiplicative Boolean convolution power} $\eta^{\utimes s}\in\cE$ by 
$$
\eta^{\utimes s} (x) := x \left( \frac{\eta(x)}{x}\right)^s = x \exp(-s u(x)). 
$$
\end{defi}

Then we generalize multiplicative free convolution to the class $\cE$. 
For $t\geq 1$, define a map $\Phi_t\colon(-\infty,0)\to(-\infty,0)$ by 
\begin{equation}\label{def Phi}
\Phi_t(x) = x \left( \frac{x}{\eta(x)}\right)^{t-1} = x \exp[(t-1)u(x)],    
\end{equation}
which is continuous and strictly increasing. Moreover, since $u$ is non-increasing, $u(-\infty) \in\R\cup\{\infty\}$ and $u(-0) = \R \cup\{-\infty\}$, and hence 
$$
\Phi_t(-\infty) = -\infty \qquad \text{and} \qquad \Phi_t(-0)=0. 
$$
Therefore, $\Phi_t$ is a homeomorphism of $(-\infty,0)$. Denote by $\omega_t$ its inverse map. We define a map $\eta^{\boxtimes t} \in\cE$ by 
\begin{equation}\label{FCP}
\eta^{\boxtimes t}(x):= \eta (\omega_t(x)). 
\end{equation}
It is not obvious if $\eta^{\boxtimes t}$ belongs to $\cE$, but it does. Since 
$$
\eta^{\boxtimes t}(x) = x \exp\left( \log \frac{\omega_t(x)}{x} - u(\omega_t(x)) \right), 
$$
it suffices to check that  $u_t(x):= -\log \frac{\omega_t(x)}{x} + u(\omega_t(x))$ is continuous and non-increasing, which is the case since  
$u_t(\Phi_t(x))=-\log \frac{x}{\Phi_t(x)} +u(x)= t u(x)$ is continuous and non-increasing. 

\begin{defi} Suppose that $\eta \in\cE$ and $t\geq1$. 
\begin{enumerate}[\rm(1)]
\item  The map $\eta^{\boxtimes t} \in\cE$ defined by \eqref{FCP} is called the {\it multiplicative free convolution power} of $\eta$. 
\item  The map $\omega_t:= \Phi_t^{-1}$ is called the {\it subordination function} of $\eta^{\boxtimes t}$ with respect to $\eta$. 
\end{enumerate}
\end{defi}

A formula for $\omega_t$ is given. The proof of \cite[Theorem 2.6(3)]{BB05} is available without a change.  
\begin{prop}\label{prop omega} For $t\geq1$, $\eta \in \cE$ and $x <0$, 
$$
\omega_t(x) = \eta^{\boxtimes t}(x) \left(\frac{x}{\eta^{\boxtimes t}(x)}\right)^{1/t}. 
$$
\end{prop}
\begin{proof} The formula follows just by substituting $\omega_t(x)$ into \eqref{def Phi} and using the identities $\Phi_t(\omega_t(x))=x$ and $\eta(\omega_t(x))=\eta^{\boxtimes t}(x)$.  
\end{proof}
From the above expression, the subordination function $\omega_t, t \geq1$ belongs to $\cE$ since  
$$
\omega_t = (\eta^{\boxtimes t})^{\utimes \frac{t-1}{t}}. 
$$
Note that $\Phi_t, t >1$ does not belong to $\cE$ by definition unless $u$ is constant.

\begin{prop}
\begin{enumerate}[\quad\rm(1)]
\item\label{PowerB} Suppose that $s, t\geq 0$ and $\eta\in\cE$. Then 
$$
(\eta^{\utimes s})^{\utimes t}=\eta^{\utimes st},\qquad \eta^{\utimes 1}=\eta.  
$$
\item\label{PowerF} Suppose that $s, t\geq 1$ and $\eta\in\cE$. Then 
$$
(\eta^{\boxtimes s})^{\boxtimes t}=\eta^{\boxtimes st},\qquad \eta^{\boxtimes 1}=\eta. 
$$
\end{enumerate}
\end{prop}
\begin{proof}
\eqref{PowerB} follows by definition. 

\eqref{PowerF} Let 
\begin{equation}\label{eq51}
\Phi_{s,t}: = x \left( \frac{x}{\eta^{\boxtimes s}(x)}\right)^{t-1} 
\end{equation}
and $\omega_{s,t}$ be its inverse. Then $(\eta^{\boxtimes s})^{\boxtimes t} = \eta^{\boxtimes s}\circ \omega_{s,t} = \eta \circ \omega_s \circ \omega_{s,t}$. Thus the claim is equivalent to $\omega_{st}=\omega_s \circ \omega_{s,t}$, which is also equivalent to 
$
\Phi_{s,t}= \Phi_{s t}\circ \omega_s. 
$
This identity follows from the calculation 
\begin{equation}\label{}
\begin{split}
\Phi_{s t}(\omega_s(x))
&= \omega_s(x)\left(\frac{\omega_s(x)}{\eta(\omega_s(x))}\right)^{s t -1} = \omega_s(x) \left(\frac{\omega_s(x)}{\eta^{\boxtimes s}(x)}\right)^{s t-1} \\
&= \eta^{\boxtimes s}(x) \left(\frac{x}{\eta^{\boxtimes s}(x)}\right)^{\frac{1}{s}}   \left(\frac{x}{\eta^{\boxtimes s}(x)}\right)^{\frac{s t -1}{s}}  \\
&= x  \left(\frac{x}{\eta^{\boxtimes s}(x)}\right)^{t-1}  =\Phi_{s,t}(x), 
 \end{split}
\end{equation}
where Proposition \ref{prop omega} was used on the third equality. 
\end{proof}

The following result extends \cite[Proposition 4.13]{AH13} with a slightly different formulation. The same proof is available, but we give a simpler proof.  

\begin{prop}\label{commutation} For $\eta \in\cE$ and $p \geq0,q \geq 1$, the commutation relation
$$
(\eta^{\utimes p})^{\boxtimes q} =(\eta^{\boxtimes q'})^{\utimes p'}
$$
holds, where $p' := p q/(1-p+p q)$ and $q':= 1-p+p q$. Note that $p' \geq 0$ and $q'\geq 1$. 
\end{prop}
\begin{proof} Denote by $\omega_{\eta,t}$ the subordination function of $\eta^{\boxtimes t}$ wrt $\eta$, and $\Phi_{\eta,t}$ similarly. 
Firstly, note that the identity 
\begin{equation}
\omega_{\eta^{\sutimes p}, q} = \omega_{\eta, q'}
\end{equation}
holds since 
\begin{equation}\label{}
\begin{split}
\Phi_{\eta^{\sutimes p}, q}(x) 
&= x \left( \frac{x}{\eta^{\utimes p}(x)}\right)^{q-1} = x \left( \frac{x}{\eta(x)}\right)^{q p-p} \\
&=  x \left( \frac{x}{\eta(x)}\right)^{q'-1} = \Phi_{\eta,q'}(x).  
\end{split}
\end{equation}
Therefore, 
\begin{equation}\label{eq commutation}
\begin{split}
 (\eta^{\utimes p})^{\boxtimes q}(x) 
 &=  \eta^{\utimes p}(\omega_{\eta^{\sutimes p}, q}(x)) =  \eta^{\utimes p}(\omega_{\eta, q'}(x)) \\
 &= \omega_{\eta, q'}(x) \left(\frac{\eta(\omega_{\eta, q'}(x))}{\omega_{\eta, q'}(x)}\right)^{p} = \omega_{\eta, q'}(x) \left(\frac{\eta^{\boxtimes q'}(x)}{\omega_{\eta, q'}(x)}\right)^{p}. 
\end{split}
\end{equation}
Proposition \ref{prop omega} then yields 
\begin{equation}\label{}
\begin{split}
\eqref{eq commutation}
 &=\eta^{\boxtimes q'}(x) \left(\frac{x}{\eta^{\boxtimes q'}(x)}\right)^{1/q'} \left(\frac{\eta^{\boxtimes q'}(x)}{x}\right)^{p/q'} \\
 &= x  \left(\frac{\eta^{\boxtimes q'}(x)}{x}\right)^{\frac{q'-1+p}{q'}}  =x  \left(\frac{\eta^{\boxtimes q'}(x)}{x}\right)^{p'} = (\eta^{\boxtimes q'})^{\utimes p'}(x), 
\end{split}
\end{equation}
the conclusion. 
\end{proof}

\begin{defi}
\begin{enumerate}[\rm(1)]
\item  A family $\{\eta_t\}_{t\geq0} \subset \cE$ is called a {\it  $\boxtimes$-convolution semigroup} if $\eta_0=\Id$ and $\eta_t^{\boxtimes s}=\eta_{s t}$ for all $s\geq1, t\geq 0$. 

\item We say that $\eta\in\cE$ {\it embeds into a $\boxtimes$-convolution semigroup} if there exists a $\boxtimes$-convolution semigroup $\{\eta_t\}_{t\geq0} \subset \cE$ such that $\eta_1=\eta$. Note that $\eta_t=\eta^{\boxtimes t}$ for all $t\geq1$. 
\end{enumerate}
\end{defi}

\begin{prop}\label{unique}
Embedding of a given $\eta\in\cE$ into a $\boxtimes$-convolution semigroup is unique, if exists.  
\end{prop}
\begin{proof} For clarity we denote by $\omega_{\eta, t}, t \geq1$ the subordination function of $\eta^{\boxtimes t}$ with respect to $\eta$. 

Suppose that $\{\eta_t\}_{t\geq0} \subset \cE$ is a $\boxtimes$-convolution semigroup into which $\eta$ embeds. 
For $0<t\geq1$, the map $\eta_t$ is given by $\eta^{\boxtimes t}$ and hence is unique. For $t<1$, we have $\eta_t^{\boxtimes 1/t} = \eta$ by definition, and so 
$$
\eta_t \circ \omega_{\eta_t, 1/t} =\eta, 
$$
where 
$$
\omega_{\eta_t,1/t}(x)= \eta_t^{\boxtimes 1/t}(x) \left(\frac{x}{\eta_t^{\boxtimes 1/t}(x)}\right)^{t} = \eta(x) \left(\frac{x}{\eta(x)}\right)^{t} =\eta^{\utimes(1-t)}
$$
 by Proposition \ref{prop omega} and $\eta_t^{\boxtimes 1/t} = \eta$. This implies that $\eta_t = \eta \circ \omega^{-1}_{\eta_t,1/t}$ only depends on $\eta$, showing the uniqueness of $\eta_t$ for $0<t<1$. 
\end{proof}

Thanks to the uniqueness, we may write $\eta_t=\eta^{\boxtimes t}$ for $t\geq0$ without ambiguity, when $\eta$ embeds into a $\boxtimes$-convolution semigroup $\{\eta_t\}_{t\geq0} \subset \cE$.

\begin{defi} We define a map $\M\colon \cE\to\cE$ by 
$$
\M(\eta):= (\eta^{\boxtimes 2})^{\utimes \frac{1}{2}}. 
$$
This map is called the {\it multiplicative Boolean-to-free Bercovoci-Pata map}, which generalizes the injective map (but not surjective) defined in \cite{AH13} from the class of probability measures on $[0,\infty)$ to the class of $\boxtimes$-ID measures.  
\end{defi}

\begin{prop}\label{M}  For any $\eta\in\cE$, the map $\M(\eta)$ embeds into a $\boxtimes$-convolution semigroup and 
$$
(\M(\eta)^{\boxtimes t})^{\utimes 1/t} = (\eta^{\boxtimes (1+t)})^{\utimes \frac{1}{1+t}}, \qquad t>0.  
$$
\end{prop}
\begin{proof} Define 
$$
\xi_t:=(\eta^{\boxtimes (1+t)})^{\utimes \frac{t}{1+t}},\qquad t\geq0.  
$$
Then $\xi_0=\Id$, $\xi_1=\M(\eta)$ and for $s\geq1, t>0$, 
$$
\xi_{t}^{\boxtimes s} = ((\eta^{\boxtimes (1+ t)})^{\utimes \frac{t}{1+t}})^{\boxtimes s} = ((\eta^{\boxtimes (1+ t)})^{\boxtimes \frac{1+s t}{1+t}})^{\utimes \frac{s t}{1+s t}} = \xi_{s t}, 
$$
where Proposition \ref{commutation} was used for $p = \frac{t}{1+t}$ and $q=s$. 
Again Proposition \ref{commutation} for  $p=1/2$ and $q=t$ yields that, for $t\geq1$ 
$$
\M(\eta)^{\boxtimes t} =  ((\eta^{\boxtimes 2})^{\utimes \frac{1}{2}})^{\boxtimes t} = ((\eta^{\boxtimes 2})^{\boxtimes \frac{1+t}{2}})^{\utimes \frac{t}{1+t}} =\xi_t, 
$$
and hence  $\M(\eta)$ embeds into the $\boxtimes$-convolution semigroup $\{\xi_t\}_{t\geq0}$. We therefore may write $\xi_t=\M(\eta)^{\boxtimes t}$ for $t\geq0$, and then  
$$
(\M(\eta)^{\boxtimes t})^{\utimes 1/t} = \xi_t^{\utimes 1/t} = (\eta^{\boxtimes (1+t)})^{\utimes \frac{1}{1+t}},\qquad t >0,  
$$
the conclusion. 
\end{proof}

From now on we assume the analyticity of $\eta\in\cE$ for a later application to the limit theorem. Let $\An(S)$ be the set of analytic functions in an open set $S\subset \C$. 

\begin{lem}\label{continuation} Let $\Omega:= (\ii\C^+) \cup \C^-$ and $0<\kappa < 1 < \lambda <\infty$. Suppose that $u \in C(\Omega \cup [\kappa,\lambda]) \cap\An(\Omega)$ such that $u([\kappa,\lambda])\subset \C^+$ and $u|_{(-\infty,0)}$ is a non-increasing map from $(-\infty,0)$ into itself. Let $\eta \in\cE$ be the map associated to $u|_{(-\infty,0)}$. Then for every $\ep\in(0,(\lambda -\kappa)/2)$ there exists $\delta>0$ such that for every $t\in(0,\delta)$ the map $(\eta^{\boxtimes (1+t)})^{\utimes \frac{1}{1+t}}$ extends to a function in $C(\overline{\Omega}_\ep)\cap \An(\Omega_\ep)$, where  
$$
\Omega_\ep:= \{z \in \Omega: \text{\rm dist}(z, \Omega^c) > \ep, |z|< \ep^{-1} \} \cup \{z\in \Omega: \Re(z) \in(\kappa+\ep, \lambda -\ep), |z|<\ep^{-1} \} 
$$
and $\overline{\Omega}_\ep$ is its closure. 
 Moreover, $(\eta^{\boxtimes (1+t)})^{\utimes \frac{1}{1+t}}(z) = \eta(z)(1+O(t))$ as $t\downarrow0$ uniformly on $\overline{\Omega}_\ep$. 
\end{lem}
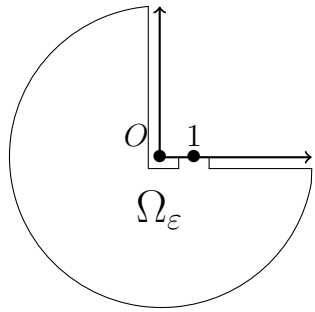
\begin{figure}[hbpt]
\begin{center}
 \begin{tikzpicture}[scale = 1]
 \draw[thick, ->] (0,0) -- (2,0) node [below] {};
\draw[thick, ->] (0,0) -- (0,2) node [left] {};
    \draw (-0.15,-0.15)-- (-0.15,2) arc (95:350:2) -- (2,-0.15) -- (0.65,-0.15) -- (0.65,0) -- (0.25,0) -- (0.25,-0.15) -- (-0.15,-0.15);  
       \coordinate (P1) at (0,0);
       \node at (0,0)  {$\bullet$};
       \node at (0,0) [above left]  {$O$};
       \node at (0.45,0)  {$\bullet$};
       \node at (0.45,0) [above]  {$1$};
       \node at (0,-0.7) {\Large$\Omega_\ep$}; 
\end{tikzpicture}
\end{center}
\end{figure}
\begin{proof} 
The functions $\eta, \Phi_{1+t}$ extend to $\An(\Omega)\cap C(\Omega \cup [\kappa,\lambda])$ by 
$$
\eta(z)=z \exp[-u(z)], \qquad \Phi_{1+t}(z) = z\exp[t u(z)]. 
$$
Since $u$ is bounded on $\Omega_{\ep/2}$, then 
\begin{equation}\label{AS-Phi}
\Phi_{1+t}(z) = z(1 +O(t))
\end{equation}
 uniformly on $z \in \Omega_{\ep/2}$. The assumption $\inf \Im(u([\kappa,\lambda]))  >0$ implies that 
\begin{equation}\label{AS-Im}
\Im(\Phi_{1+t}(x)) =x \exp[t \Re(u(x))] \sin [t \Im(u(x))] >0
\end{equation} 
for all $0<t< \pi [\sup\Im(u([\kappa,\lambda]))]^{-1}$ and all $x\in[\kappa,\lambda]$. By \eqref{AS-Phi} and \eqref{AS-Im}, for sufficiently small $t>0$ the curve $\Phi_{1+t}|_{\partial \Omega_{\ep/2}}$ surrounds every point of an open neighborhood $N_t$ of $\overline{\Omega}_{\ep}$ exactly once. Hence its right inverse $\omega_{1+t}$ (i.e.\ $\Phi_{1+t}\circ\omega_{1+t}=\Id$) can be defined as an injective analytic map from $N_t$ into $\Omega_{\ep/2}$. Thus the map
$$
\eta^{\boxtimes (1+t)} = \eta\circ \omega_{1+t} \in \An(\Omega_{\ep}) \cap C(\overline{\Omega}_{\ep})
$$ 
can be defined, which equals the original map $\eta^{\boxtimes (1+t)}\in\cE$ on the negative half-line. Moreover, plugging $\omega_{1+t}(z)$ into \eqref{AS-Phi} shows that $\omega_{1+t}(z) = z(1+O(t))$ uniformly on $\overline{\Omega}_{\ep}$, and hence 
$$
\eta^{\boxtimes (1+t)}(z) =\eta(\omega_{1+t}(z))= \eta(z)+O(t) 
$$ 
uniformly on $\overline{\Omega}_{\ep}$. Note that the last estimate may be expressed in the form $\eta(z)(1+O(t))$ since there exists a number $c>0$ such that $c^{-1} \leq |\eta(z)| \leq c$ on $\overline{\Omega}_\ep$. Finally, some calculus shows
$$
(\eta^{\boxtimes (1+t)})^{\utimes\frac{1}{1+t}}(z) = z \!\left[\frac{\eta(z)}{z}(1+O(t))\right]^{\frac{1}{1+t}} = \eta(z)(1+ O(t))
$$ 
uniformly on $\overline{\Omega}_\ep$. 
\end{proof}

\subsection{Proof of Theorem \ref{LCF}}
In this section we apply the general framework of convolutions to the limit theorem. 
Suppose that $\mu$ is $\boxtimes$-ID on $(0,\infty)$ whose $\Sigma$-transform has the expression \eqref{eq FLK}. 
Then define 
\begin{equation}\label{eta-def}
\eta(x):= \frac{x}{\Sigma_\mu(x)} = x \exp(-\vv_\mu(x)),  \qquad x<0. 
\end{equation} 
The correspondence $\mu \mapsto \eta$ is a generalization of the multiplicative Bercovici-Pata map from Boolean to free \cite{AH13} which is not surjective. 
We can show that $\eta\in\cE$, but $\eta$ may not be the $\eta$-transform of a probability measure. Moreover, the map $\eta$ is not even injective on $(-\infty,0)$ in general, so 
it seems not easy to define a $\Sigma$- or $S$-transform of $\eta$.  
This is why the previous section has investigated convolution operations for maps on $(-\infty,0)$ without using $\Sigma$-transforms. 

The following result shows that the map $\M$ is a generalization of the multiplicative Bercovici-Pata map from Boolean to free (denoted by $\M_1$ in \cite{AH13}). 

\begin{lem} Under the above notation, the identity $
\M(\eta) = \eta_\mu
$
holds. 
\end{lem}
\begin{proof} Recall that $\eta_\mu$ is a homeomorphism of $(-\infty,0)$ since $\mu(\{0\})=0$. Since $\eta(x) = \frac{x^2}{\eta_\mu^{-1}(x)}$, we obtain 
$$
\Phi_2(x) = \frac{x^2}{\eta(x)} = \eta_\mu^{-1}(x), 
$$
and hence 
$$
\omega_2 = \eta_\mu. 
$$
Therefore, 
$$
\eta^{\boxtimes 2}(x) = \eta(\omega_2(x)) = \frac{\eta_\mu(x)^2}{x}, 
$$
and so 
$$
(\eta^{\boxtimes 2})^{\utimes \frac{1}{2}}(x) = x \left( \frac{\eta^{\boxtimes 2}(x)}{x}\right)^\frac{1}{2} = x\left( \frac{[-\eta_\mu(x)]^2}{(-x)^2}\right)^\frac{1}{2}=\eta_\mu(x). 
$$
\end{proof}

Under the setting of \eqref{eta-def}, since $\mu$ embeds into the $\boxtimes$-convolution semigroup $\{\mu^{\boxtimes t}\}_{t\geq0}$ of probability measures on $(0,\infty)$, the map $\eta_\mu$ embeds into the $\boxtimes$-convolution semigroup $\eta_{\mu^{\boxtimes t}}$. This convolution semigroup may be written as $(\eta_\mu)^{\boxtimes t}$ and it coincides with $(\eta^{\boxtimes (1+t)})^{\utimes \frac{t}{1+t}}$ by Proposition \ref{M}. Now, we have the identity 
\begin{equation}\label{key eq}
\eta_{\mu^{\boxtimes t}}= ((\eta^{\boxtimes (1+t)})^{\utimes \frac{1}{1+t}})^{\utimes t}. 
\end{equation}
We have shown that the function $(\eta^{\boxtimes (1+t)})^{\utimes \frac{1}{1+t}}$ is close to $\eta$ up to $O(t)$ when $t$ is small. This estimate and \eqref{key eq} enable us to reduce the convergence problem of a MFLP to the Boolean case. 

\begin{proof}[Proof of Theorem \ref{LCF}] By assumption, there exist $0<\kappa<1<\lambda<\infty$ such that $\vv_\mu$ extends to a continuous function on $\Omega \cup [\kappa,\lambda]$ and such that $\vv_\mu([\kappa,\lambda]) \subset \C^+$ as required in Lemma \ref{continuation}. Fixing $\ep >0$ such that $\lambda+\ep <1 <\kappa-\ep$, Lemma \ref{continuation} shows that the map $\eta_t:=(\eta^{\boxtimes (1+t)})^{\utimes \frac{1}{1+t}}$ continuously extends to $\overline{\Omega}_\ep$ for small $t>0$. 
Therefore, Eq.\ \eqref{key eq} implies that, for $1/z \in \Omega_\ep$, 
\begin{equation}\label{}
\begin{split}
G_{\mu^{\boxtimes t}}(z) 
&= \frac{1}{z - z \eta_{\mu^{\boxtimes t}}(1/z) } = \frac{1}{z - z \eta_t^{\utimes t}(1/z) } \\
&=\frac{1}{z - (z \eta_t(1/z))^t } = \frac{1}{z - \exp[ -t \vv_\mu(1/z)](1+O(t))^t },  
\end{split}
\end{equation}
where the asymptotic behavior of Lemma \ref{continuation} was used on the last equality. Notice that the map $z\mapsto (z \eta_t(1/z))^t$ may not be the principal value in $\C^-$, but is defined as the (unique) continuous extension of the real-valued map on $(-\infty,0)$.  
Therefore, the density of $\mu^{\boxtimes t}$ is given by 
\begin{equation}\label{}
\begin{split}
\frac{d \mu^{\boxtimes t}}{dx} 
&= -\frac{1}{\pi}\Im [G_{\mu^{\boxtimes t}}(x+\ii0)] \\
&=  -\frac{1}{\pi}\Im\!\left(\frac{1}{x - \exp[-t \vv_\mu(1/x -\ii0)](1+o(t)) } \right)
\end{split}
\end{equation}
in a neighborhood of 1 unless the denominator is zero. 

Take a compact subset $K$ of $(0,\infty)$. The density of  $(\mu^{\boxtimes t})^{1/t}$ is given by 
\begin{equation}
\begin{split}
\frac{d (\mu^{\boxtimes t})^{1/t}}{d x} 
&= -\frac{t}{\pi x}\Im\!\left(\frac{1}{1 - x^{-t} \exp[-t \vv_\mu(x^{-t} -\ii0)](1+o(t))  }\right) \\
&= -\frac{t}{\pi x}\Im\!\left(\frac{1}{1 - (1-t \log x+o(t)) (1-t \vv_\mu(1)+o(t))(1+o(t)) }\right) \\
&= -\frac{1}{\pi x}\Im\!\left(\frac{1}{\log x+\vv_\mu(1)+o(1) }\right)
\end{split}
\end{equation}
as $t\downarrow0$ uniformly on $K$. This computation shows that $d (\mu^{\boxtimes t})^{1/t}/ d x$ exists on $K$ if $t$ is small enough (since the denominator is not zero),  
and converges to
$$
\frac{1}{\pi x} \cdot \frac{\gamma}{(\log x -\beta)^2 + \gamma^2} \quad \text{as} \quad t\downarrow0
$$
uniformly on $K$. 
\end{proof}

We do not know if this method somehow extends to other log free stable distributions. The main difficulty is that most of the log free stable distributions do not have explicit densities; they only have densities described by some implicit functions. In the above analysis, it is not obvious how these implicit densities appear in the limit.

\section{Unitary multiplicative L\'evy processes at small time}\label{S7}

In this section we will find limit distributions for unitary MFLPs and MCLPs at small time. We do not discuss unitary MBLPs because of some technical difficulty. We want to study the convergence in law of the unitary process
\begin{equation}
b(t) (U_t)^{a(t)},  \quad \text{as~} t\downarrow0, 
\end{equation}
where $\{U_t\}_{t\geq0}$ is a unitary MFLP and $a\colon (0,\infty) \to \N$ and $b\colon (0,\infty) \to \T$ are functions. Note that the function $a$ is assumed to be $\N$-valued, or at least $\Z$-valued, because non-integral powers $z^p$ cannot be continuously defined on $\T$. In terms of probability measures, our aim is to obtain weak limits of
\begin{equation}\label{UMFLP}
\Rot_{b(t)}(\mu^{\boxtimes t})^{a(t)},  \quad \text{as~} t\downarrow0, 
\end{equation}
where $\{\mu^{\boxtimes t}\}_{t\geq0}$ is a weakly continuous $\boxtimes$-convolution semigroup on $\T$ such that $\mu_0=\delta_1$. 
\begin{rem}
We cannot formulate a limit theorem for unitary MFLPs at large time. In order to do so we need to consider the situation where $a(t)\to 0$ as $t \to \infty$, but such is impossible for $\N$-valued functions. 
\end{rem}

Before going to the general case let us analyze the important example of unitary free Brownian motion for which we have an explicit description in terms of moments.
\begin{exa}\label{limit semicircle}
Let $\{U_t\}_{t\geq0}$ be standard unitary free BM. The $m$-th moment of $U_t$ is calculated by Biane \cite{Bia97}: 
$$ \bE[U_t^m]= e^{-\frac{m t}{2}}\sum^{m-1}_{k=0} (-1)^k\frac{t^k}{k!} m^{k-1} \binom{m}{k+1}, \qquad m \geq1.  
$$
If we take $m = n [1/\sqrt{t}]$ for a fixed $n \in\N$ then as $t$ tends to 0 we have 
\begin{equation*}
\begin{split}
\bE\left[\left(U_{t}^{[1/\sqrt{t}]}\right)^n\right] \sim e^{-\frac{n \sqrt{t}}{2}}\sum^{n [1/\sqrt{t}]-1}_{k=0} (-1)^k t^k \frac{(n t^{-1/2})^{2k}}{k! (k+1)!}   \to \sum^{\infty}_{k=0} (-1)^k \frac{n^{2k}}{k! (k+1)!} = \frac{J_1(2n)}{n},  
\end{split}
\end{equation*}
where $J_1$ is the Bessel function of the 1st kind. Let $S$ be a semicircular random variable with mean 0 and variance 1. Then, it is well known that the characteristic function is given by 
$$
\bE[e^{\ii \gamma S}] =  \frac{J_1(2\gamma)}{\gamma}, \qquad \gamma\in\R. 
$$
So we have proved that 
$$
(U_t)^{[1/\sqrt{t}]} \tolaw e^{\ii S}, \qquad t\downarrow0.   
$$
\end{exa}

In order to discuss the general case, it is instructive to understand the classical version of \eqref{UMFLP} in which $\{\mu^{\boxtimes t}\}_{t\geq0}$ is replaced by $\{\mu^{\circledast t}\}_{t\geq0}$. Let $\{X_t\}_{t\geq0}$ be an ACLP on $\R$ such that $X_0=0$, and let $U_t =e^{\ii X_t}$. Then $\{U_t\}_{t\geq0}$ is a unitary MCLP and the identity  
\begin{equation}\label{UMLP}
b(t) (U_t)^{a(t)} = e^{\ii [a(t) X_t + \arg b(t) ]} 
\end{equation}
holds, where $ \arg b(t)$ is defined modulo $2\pi$. If we take the ACLP $\{X_t\}$ and functions $a$ and $b$ in such a way that $a(t) X_t + \arg b(t)$ converges in law to a stable random variable $X$, then the law of \eqref{UMLP} converges in law to $e^{\ii X}$.  For example we may trivially take $\{X_t\}_{t\geq0}$ to be a stable process! This argument shows that the set of possible limit distributions contains all the laws of $e^{\ii X}$, where $X$ is a stable random variable. Moreover, it is easy to see that the Haar measure can appear in the limit. 
The authors do not know whether other distributions appear in the limit.

A similar idea works for unitary MFLPs. The wrapping map $W$ is useful to establish a transfer principle from additive convolutions to multiplicative ones, provided that we restrict to the class $\cL$ of probability measures (see Section \ref{exponential}).

\begin{thm}\label{thm UMFLP0}
Let $\mu$ be an $\boxtimes$-ID distribution on $\T$ with free generating pair $(\gamma,\sigma)$ and define 
\begin{align*}
\overline{\sigma}^+(x)= \int_{(x,\pi)} \theta^{-2} \,d\sigma(\theta), &\qquad \overline{\sigma}^-(x)= \int_{(-\pi,-x)} \theta^{-2} \,d\sigma(\theta), \\
\overline{\sigma}(x)=\overline{\sigma}^+(x)&+\overline{\sigma}^-(x),  \qquad 0<x<\pi, 
\end{align*}
where the measure $\sigma$ on $\T$ is identified with the measure on $[-\pi, \pi)$. 
\begin{enumerate}[\rm(1)]
\item If the function $x\mapsto \sigma((-x,x))$ is slowly varying as $x\downarrow0$ then there exist functions $a\colon (0,\infty) \to \N$ and $b\colon(0,\infty) \to \T$ such that \eqref{UMFLP} weakly converges to $W(\bff_{2,1/2})$. 
\item Let $(\alpha,\rho)\in \Ad, \alpha \neq2$.  If the function $ \overline{\sigma}(x)$ is regularly varying with index $-\alpha$ as $x\downarrow0$ and if 
$$
\lim_{x\downarrow0} \frac{\overline{\sigma}^+(x)}{\overline{\sigma}(x)} = 
\begin{cases}
\frac{1}{2} \left(1 + \frac{\tan(\rho-\frac{1}{2})\alpha\pi}{\tan\frac{\alpha \pi}{2}}\right), &\text{if}~\alpha \neq 1, \\
\rho, &\text{if}~ \alpha =1,  
\end{cases}
$$ 
then there exist functions $a\colon (0,\infty) \to \N$ and $b\colon(0,\infty) \to \T$ such that \eqref{UMFLP} weakly converges to $W(\bff_{\alpha,\rho})$. 
\end{enumerate}
Similar statements hold for the classical case. 
\end{thm}
\begin{proof} Let $(\xi,\tau)$ be a pair defined by \eqref{eq gamma} and \eqref{ID wrap}, and let $\tilde\mu:= \mu_\boxplus^{\xi,\tau}$. Then $\tilde \mu$ is a pre-image of $\mu$ by the map $W|_\cL$ from Proposition \ref{prop:wrapping-free}. The measure $\tau$ satisfies the assumption of Theorem \ref{DA at 0}, which implies that there exist functions $A(t), B(t)>0$ such that $A(t)\to \infty$ and $\D_{A(t)}(\tilde\mu ^{\boxplus t})\boxplus \delta_{B(t)} \wto \bff_{\alpha,\rho}$.  Since 
\begin{equation}
\D_{[A(t)]}(\tilde\mu ^{\boxplus t})\boxplus \delta_{B(t) [A(t)]/ A(t)} =\D_{[A(t)]/A(t)}\left[ \D_{A(t)}(\tilde\mu ^{\boxplus t})\boxplus \delta_{B(t)} \right]  \wto \bff_{\alpha,\rho}, 
\end{equation}
we may a priori assume that $A(t)$ is $\N$-valued. Then, by Proposition \ref{homomorphism} we have 
\begin{equation}\label{Wrap-affine}
W(\D_{A(t)}(\tilde\mu^{\boxplus t})\boxplus \delta_{B(t)}) = \Rot_{e^{-\ii B(t)}}\left[ (W(\tilde\mu)^{\boxtimes t})^{A(t)} \right] =\Rot_{e^{-\ii B(t)}}\left[ (\mu^{\boxtimes t})^{A(t)} \right], 
\end{equation}
which weakly converges to $W(\bff_{\alpha,\rho})$. This shows that we can take $a(t) = A(t)$ and $b(t)= e^{-\ii B(t)}$ such that \eqref{UMFLP} converges to $W(\bff_{\alpha,\rho})$. 

The proof for the classical case is similar; one only needs to use Proposition \ref{prop:wrapping-classical} instead of Proposition \ref{prop:wrapping-free}, and replace free objects by the corresponding classical ones.  
\end{proof}
\begin{rem}  Note that the measure $\D_{A(t)}(\tilde\mu^{\boxplus t})$ above may not belong to $\cL$, since the class $\cL$ is not closed under dilation. Due to this, the converse statement of Theorem \ref{thm UMFLP0} cannot be proved. Also, we cannot prove a similar statement for the Boolean case: since $\D_{A(t)}((\mu_\uplus^{\xi,\tau})^{\uplus t})$ may not belong to $\cL$ and $\uplus \delta_{B(t)}$ may not be the usual shift (cf.\ \eqref{equal-convolutions}), we do not know how to compute $W(\D_{A(t)}((\mu_\uplus^{\xi,\tau})^{\uplus t})\uplus \delta_{B(t)})$. 
\end{rem}

\begin{cor}\label{cor UMFLP}
The set of possible limit distributions of \eqref{UMFLP} contains the set $\{W(\mu): \mu \text{~is free stable}\}$. A similar statement holds for the classical case. 
\end{cor}

\begin{exa}
Let $\lambda_{\alpha,\rho}$ be the $\boxplus$-ID distribution with the Voiculescu transform 
\begin{align}
\varphi_{\lambda_{\alpha,\rho}}(z) = \varphi_{\bff_{\alpha,\rho}}(\tan z)=
\begin{cases}
 - e^{\ii \alpha\rho\pi}  \left(\tan z\right)^{1-\alpha}, & \alpha \neq1, \\
 -\ii \rho \pi - (1-2\rho) \log \tan z, & \alpha =1, 
\end{cases}
\end{align}
where $(\alpha,\rho)$ is admissible. Since $\tan z$ maps $\C^+$ into itself, and $\tan (\ii y) \to \ii$ as $y\to \infty$, those functions have Pick--Nevanlinna representations of the form \eqref{FLK} and hence  by Theorem \ref{thmBV93} such a $\boxplus$-ID distribution $\lambda_{\alpha,\rho}$ exists. Furthermore, $\varphi_{\lambda_{\alpha,\rho}}$ is a periodic function with respect to $2\pi$ shifts, and hence $\lambda_{\alpha,\rho}\in\cL.$ 
If we let $\mu_t=\D_{a(t)} (\lambda_{\alpha,\rho}^{\boxplus t})$ then, for $\alpha \neq1$,  
\begin{equation}
\varphi_{\mu_t}(z)= -t a(t) \varphi_{\lambda_{\alpha,\rho}}(z/a(t))= -t a(t) e^{\ii \alpha\rho\pi}  \left( \tan \frac{z}{a(t)}\right)^{1-\alpha}
\end{equation}
Supposing $a(t)\to\infty$ we have that for $t$ small enough  
\begin{equation}
\varphi_{\mu_t}(z)\sim -t a(t)^\alpha e^{\ii \alpha\rho\pi} z^{1-\alpha}.
\end{equation}
Taking $a(t)=[t^{-1/\alpha}]$ we obtain that $\varphi_{\mu_t}(z)\to- e^{\ii \alpha\rho\pi}  {z^{1-\alpha}}$.
This implies  that $\mu_t \to \mathbf{f}_{\alpha,\rho}$ and hence 
\begin{equation}
(W(\lambda_{\alpha,\rho})^{\boxtimes t})^{[t^{-1/\alpha}]} = W(\D_{a(t)}(\lambda_{\alpha,\rho}^{\boxplus t}))\wto W(\mathbf{f}_{\alpha,\rho}).
\end{equation}
Similarly, setting $\tilde\mu_t=\D_{\tilde a(t)} (\lambda_{1,\rho}^{\boxplus t}) \boxplus \delta_{\tilde b(t)}$ where $\tilde a(t) \to \infty$, we obtain
\begin{equation}
\varphi_{\tilde\mu_t}(z)\sim -t \tilde a(t)  \ii\rho\pi - (1-2\rho)t \tilde a(t) \log z +(1-2\rho) t \tilde a(t) \log \tilde a(t) + \log \tilde b(t).
\end{equation}
Therefore taking $\tilde a(t) = [1/t]$ and $\tilde b(t) = t^{1-\rho}$ yields the convergence $\varphi_{\tilde\mu_t}(z) \to \varphi_{\bff_{1,\rho}}(z)$. This shows the weak convergence
\begin{equation}
\Rot_{\exp(-\ii t^{1-\rho})}(W(\lambda_{1,\rho})^{\boxtimes t})^{[1/t]} \wto W(\mathbf{f}_{1,\rho}).
\end{equation}
\end{exa}

In the unitary case, the Haar measure can appear as a limit distribution. For example if the measure $\mu$ itself is the Haar measure, then the measure  \eqref{UMFLP} is the Haar measure at any time. 
\begin{prob}
Is the set $\{W(\mu)\mid \mu \text{~is free stable}\} \cup\{\text{Haar measure, delta measures}\}$ the only possible limits of \eqref{UMFLP}?
\end{prob}

\section*{Acknowledgements} O.A.\ was supported by CONACYT grant 222667. T.H.\ is supported by JSPS Grant-in-Aid for Young Scientists (B) 15K17549 and  by JSPS and MAEDI Japan--France Integrated Action Program (SAKURA).  The authors are grateful to Victor Rivero who informed them of the references \cite{MM08,MM09}.


~ ~\\
Department of Probability and Statistics, Centro de Investigaci\'on en Matem\'aticas, Guanajuato, M\'exico. \emph{E-mail address}  octavius@cimat.mx

~~\\
Department of Mathematics, Hokkaido University, Kita 10, Nishi 8, Kita-ku, Sapporo 060-0810, Japan.  \emph{E-mail address}  thasebe@math.sci.hokudai.ac.jp


\begin{thebibliography}{100}

\bibitem[AS70]{AS70} M.\ Abramowitz and I.A.\ Stegun, \emph{Handbook of Mathematical Functions with Formulas, Graphs, and Mathematical Tables}, National Bureau of Standards, Washington, 1970. 

\bibitem[AA17]{AA} M.\ Anshelevich and O.\ Arizmendi, The exponential map in non-commutative probability, Int.\ Math.\ Res.\ Notices.\ {\bf 17} (2017),  5302--5342.

\bibitem[AH16]{AH16} O.\ Arizmendi and T.\ Hasebe, Classical scale mixtures of Boolean stable laws, Trans.\ Amer.\ Math.\ Soc.\ {\bf 368} (2016), 4873--4905.

\bibitem[AH13]{AH13} O.\ Arizmendi and T.\ Hasebe, Semigroups related to additive and multiplicative, free and Boolean convolutions, Studia Math.\ {\bf 215}, No.\ 2 (2013), 157--185.

\bibitem[BBCC11]{BBCC11} T.\ Banica, S.T.\ Belinschi, M.\ Capitane and B.\ Collins, Free Bessel laws, Canad.\ J.\ Math.\  63 (2011), no.\ 1, 3--37.


\bibitem[B-NT02]{B-NT02} O.E.\ Barndorff-Nielsen and S.\ Thorbj{\o}rnsen, Self-decomposability and L\'evy processes in free probability, Bernoulli {\bf8}(3) (2002), 323--366. 

\bibitem[BBGH12]{BBGH} R.\ Basu, A.\ Bose, S.\ Ganguly and R.\ Hazra, Spectral properties of random triangular matrices, Random Matrices: Theor.\ Appl.\ \textbf{01}, No.\ 03 (2012), 1250003, 22 pp. 

\bibitem[BB04]{BB04} S.T.~Belinschi and H.~Bercovici,  Atoms and regularity for measures in a partially defined free convolution semigroup, Math.\ Z.\ {\bf 248} (2004), No.\ 4, 665--674. 

\bibitem[BB05]{BB05} S.T.~Belinschi and H.~Bercovici, Partially defined semigroups relative to multiplicative free convolution, Int.\ Math.\ Res.\ Notices, No.\ 2 (2005), 65--101.

  \bibitem[Ber06]{Ber06} H.~Bercovici, On Boolean convolutions, Operator Theory \textbf{20}, 7--13, Theta.\ Ser.\ Adv.\ Math.\ 6, Theta, Bucharest, 2006.

\bibitem[BP99]{BP99} H.~Bercovici and V.~Pata, Stable laws and domains of attraction in free probability theory (with an appendix by Philippe Biane), Ann.\ of Math.\ (2) \textbf{149}, No.\ 3 (1999), 1023--1060.



\bibitem[BV93]{BV93} H.~Bercovici and D.~Voiculescu, Free convolution of measures with unbounded support, Indiana Univ.\ Math.\ J.\ \textbf{42}, No.\ 3 (1993), 733--773.

\bibitem[BV92]{BV92} H.~Bercovici and D.~Voiculescu, L\'{e}vy-Hin\v{c}in type theorems for multiplicative and additive free convolution, Pacific J.\ Math.\ \textbf{153}, No.\ 2 (1992), 217--248.

\bibitem[Bia97]{Bia97}  Ph.\ Biane, Free Brownian motion, free stochastic calculus and random matrices. Free probability theory (Waterloo, ON, 1995), 1--19, Fields Inst.\ Commun.\ 12, Amer.\ Math. Soc., Providence, RI, 1997. 

\bibitem[Bia98]{Bia98} Ph.\ Biane, Processes with free increments, Math.\ Z.\ {\bf 227} (1) (1998), 143--174. 

\bibitem[BGT87]{BGT87} N.H.\ Bingham, C.M.\ Goldie and J.L.\ Teugels, {\it Regular variation}, Encyclopedia of Math.\ and its Appl., 27, Cambridge University Press, Cambridge, 1987. 


\bibitem[C\'eb16]{Ceb16} G.\ C\'ebron, Matricial model for the free multiplicative convolution, Ann.\ Probab.\ {\bf 44}, No.\ 4 (2016), 2427--2478.

\bibitem[Che]{Che} D.\ Cheliotis, Triangular random matrices and biorthogonal ensembles. arXiv:1404.4730v1

\bibitem[CG08]{CG08} G.P.\ Chistyakov and F.\ G\"otze, Limit theorems in free probability theory II, Cent.\ Eur.\ J.\ Math.\ {\bf 6} (2008), no.\ 1, 87--117.

\bibitem[Dem11]{Dem11} N.~Demni, Kanter random variable and positive free stable distributions, Electron.\ Commun.\ Probab.\ {\bf 16} (2011), 137--149.

\bibitem[DH04a]{DH04a} K.\ Dykema and U.\ Haagerup, DT-operators and decomposability of Voiculescu's circular operator,  Amer.\ J.\ Math.\ \textbf{126}(1) (2004), 121--189.

\bibitem[DH04b]{DH04b} K.\ Dykema and U.\ Haagerup, Invariant subspaces of the quasinilpotent DT-operator, J.\ Funct.\ Anal.\ \textbf{209}(2) (2004), 332--366.
  
\bibitem[Fra09a]{Fra09a} U.\ Franz, Monotone and boolean convolutions for non-compactly supported probability measures, Indiana Univ.\ Math.\ J.\ {\bf 58}, No.\ 3 (2009), 1151--1186.

\bibitem[Fra09b]{Fra09b} U.\ Franz, Boolean convolution of probability measures on the unit circle, Analyse et probabilit\'es, S\'eminaires et Congr\`es {\bf 16} (2009), 83--93.

\bibitem[GK54]{GK54} B.V.~Gnedenko and A.N.~Kolmogorov, {\it Limit Distributions for Sums of Independent Random Variables}, Addison-Wesley Publishing Company, Inc., 1954.


\bibitem[HM13]{HM} U.\ Haagerup and S.\ M\"oller, The law of large numbers for the free multiplicative convolution, in: Operator Algebra and Dynamics, Springer Proceedings in Mathematics \& Statistics \textbf{58}, 2013, 157--186. 

\bibitem[HS07]{HS07} U.\ Haagerup and H.\ Schultz, Brown measures of unbounded operators affiliated with a finite von Neumann algebra, Math.\ Scand.\ \textbf{100} (2007), 209--263. 


\bibitem[Has14]{Has14} T.\ Hasebe, Free infinite divisibility for beta distributions and related ones, Electron.\ J.\ Probab.\ {\bf 19}, No.\ 81 (2014), 1--33.

\bibitem[Has10]{Has10} T.\ Hasebe, Monotone convolution semigroups, Studia Math. {\bf 200} (2010),  175--199.

\bibitem[HK14]{HK14} T.\ Hasebe and A.\ Kuznetsov, On free stable distributions, Electron.\ Commun.\ Probab.\ \textbf{19}, No.\ 56 (2014), 1--12. 

\bibitem[HS15]{HS15} T.\ Hasebe and N.\ Sakuma, Unimodality of Boolean and monotone stable distributions, Demonstr.\ Math.\ {\bf 48}, No.\ 3 (2015), 424--439. 


\bibitem[Maa92]{Maa92} H.\ Maassen, Addition of freely independent random variables, J.\ Funct.\ Anal.\ {\bf 106} (1992), 409--438. 

\bibitem[MM08]{MM08} R.\ Maller and D.M.\ Mason, Convergence in distribution of L\'evy processes at small times with self-normalization, Acta Sci.\ Math.\ (Szeged) {\bf 74}, no.\ 1-2 (2008), 315--347. 

\bibitem[MM09]{MM09} R.\ Maller and D.M.\ Mason, Stochastic compactness of L\'evy processes, High dimensional probability V: the Luminy volume, 239--257, Inst\. Math.\ Stat.\ Collect.\ {\bf 5}, Inst.\ Math.\ Statist., Beachwood, OH, 2009. 

\bibitem[M{\l}o10]{M10} W.\ M\l otkowski, Fuss-Catalan numbers in noncommutative probability, Doc.\ Math.\ \textbf{15} (2010), 939--955.

\bibitem[NS96]{NS96} A.\ Nica and R.\ Speicher, On the multiplication of free N-tuples of noncommutative random variables, Amer.\ J.\ Math.\ {\bf 118}, No.\ 4 (1996), 799--837.

\bibitem[Par67]{Par67} K.R.\ Parthasarathy,  {\it Probability Measures on Metric Spaces}, Academic Press, Inc., New York, 1967.


\bibitem[SY13]{SY13} N.\ Sakuma and H.\ Yoshida, New limit theorems related to free multiplicative convolution, Studia Math.\ \textbf{214}, No.\ 3 (2013),  251--264

\bibitem[Sat99]{Sat99} K.\ Sato, {\it L\'evy Processes and Infinitely Divisible Distributions}, Cambridge Studies in Advanced Math.\ 68, 1999.

\bibitem[\'Sni03]{S03} P.\ \'Sniady, Multinomial identities arising from the free probability, J.\ Combin.\ Theor.\ A \textbf{101}, No.\ 1 (2003), 1--19. 

\bibitem[SW97]{SW97} R.\ Speicher and R.\ Woroudi, Boolean convolution, Free Probability Theory, Ed.\ D.\ Voiculescu, Fields Inst.\ Commun.\ {\bf 12}, Amer.\ Math.\ Soc. (1997), 267--280. 

\bibitem[Tit26]{Tit26} E.C.\ Titchmarsh, Conjugate trigonometrical integrals, Proc.\ London Math.\ Soc.\ {\bf s2-24} (1) (1926), 109--130.

\bibitem[Tuc10]{T10} G.H.\ Tucci, Limits laws for geometric means of free random variables, Indiana Univ.\ Math.\ J.\ \textbf{59}(1) (2010), 1--13. 

\bibitem[Voi86]{Voi86} D.\ Voiculescu, Addition of certain non-commutative random variables, J.\ Funct.\ Anal.\ {\bf 66} (1986), 323--346.

\bibitem[Voi87]{Voi87} D.\ Voiculescu, Multiplication of certain noncommuting random variables, J.\ Operator Theory \textbf{18} (1987), 223--235.

\bibitem[Voi85]{Voi85} D.~Voiculescu, Symmetries of some reduced free product $C^\ast$-algebras, in: Operator Algebras and their Connections with Topology and Ergodic Theory, 556--588, Lecture Notes in Mathematics, Vol. 1132, Springer-Verlag, Berlin/New York, 1985.



\bibitem[Zol86]{Zol86} V.M.\ Zolotarev, {\it One-dimensional stable distributions}, Translations of Mathematical Monographs 65. Amer.\ Math.\ Soc., Providence, RI, 1986.  

\end{thebibliography}
\end{document}